\documentclass[reqno,12pt]{amsart}
  \RequirePackage{amsmath}
\RequirePackage{amssymb} \RequirePackage{amsthm}
\RequirePackage{epsfig} 
\usepackage{color}   

\def\ep{\varepsilon}

\newcommand{\D}{\mathbb{D}}

\newcommand{\Fpa}{\mathcal{F}^p_{\alpha}}

\newcommand{\Lpaom}{\mathcal{L}^p_{\alpha,\om}}

\newcommand{\Fpaom}{\mathcal{F}^p_{\alpha,\om}}
\newcommand{\Lpa}{\mathcal{L}^p_{\alpha}}

\newcommand{\N}{\mathbb{N}}

\newcommand{\C}{\mathbb{C}}

\newcommand{\Z}{\mathbb{Z}}
\newcommand{\R}{\mathbb{R}}

\newcommand{\vp}{\vp}

\newcommand{\Ap}{A^{restricted}_{p}}
\newcommand{\Ainfty}{A^{restricted}_{\infty}}
\newcommand{\Apo}{A^{restricted}_{p_0}}
\newcommand{\Aq}{A^{restricted}_{q}}
\newcommand{\Apprime}{A^{restricted}_{p'}}

\def\a{\alpha}       \def\b{\beta}        \def\g{\gamma}
           \def\e{\varepsilon} 
     \def\om{\omega}      
              
                  \def\z{\zeta}

\newtheorem{theorem}{Theorem}[section]
\newtheorem{lemma}[theorem]{Lemma}
\newtheorem{proposition}[theorem]{Proposition}

\newtheorem{corollary}[theorem]{Corollary}
\newtheorem{remark}[theorem]{Remark}

\newtheorem{lettertheorem}{Theorem}

\newtheorem{letterlemma}[lettertheorem]{Lemma}
\newtheorem{letterproposition}[lettertheorem]{Proposition}

\theoremstyle{definition}
\newtheorem{definition}[theorem]{Definition}

\newenvironment{Pf}{\noindent{\emph{Proof of}}}{$\hfill\Box$ }

\numberwithin{equation}{section}
\begin{document}
\title[Weighted Fock spaces]
{Littlewood-Paley formulas and Carleson measures for weighted Fock
spaces induced by $A_\infty$-type weights}

\author[C. Cascante]{Carme Cascante}
\address{Carme Cascante\\Departament de Matem\`{a}tica Aplicada i Analisi\\
Universitat de Barcelona\\
Gran Via 585, 08007 Barcelona\\
Spain} \email{cascante@ub.edu}

\author[J.  F\`{a}brega]{Joan  F\`{a}brega}
\address{Joan  F\`{a}brega \\Departament de Matem\`{a}tica Aplicada i Analisi\\
Universitat de Barcelona\\
Gran Via 585, 08007 Barcelona\\
Spain}
\email{joan$_{-}$fabrega@ub.edu}

\author[J. A. Pel\'aez]{Jos\'e A. Pel\'aez}
\address{Jos\'e A. Pel\'aez\\Departamento de An\'alisis Matem\'atico\\
Facultad de Ciencias\\
29071, M\'alaga\\
Spain. } \email{japelaez@uma.es}

\date{\today}

\thanks{The research of the first two authors was supported in part by
	Ministerio de Econom\'{\i}a y Competitividad, Spain,   projects MTM2014-51834-P and MTM2015-69323-REDT, and Generalitat de Catalunya, project
	2014SGR289.\newline
	The research of the third author was supported in part by
	Ministerio de Econom\'{\i}a y Competitividad, Spain,   projects
	MTM2014-52865-P, and MTM2015-69323-REDT; La Junta de Andaluc{\'i}a,
	project FQM210}



\begin{abstract}
	We obtain Littlewood-Paley formulas for  Fock spaces $\mathcal{F}^q_{\beta,\omega}$  induced by weights
 $\omega\in\Ainfty= \cup_{1\le p<\infty}A^{restricted}_{p}$, where $A^{restricted}_{p}$ is the class of weights such that
 the Bergman projection $P_\alpha$, on the classical Fock space $\mathcal{F}^2_{\alpha}$, is bounded on

$$\mathcal{L}^p_{\alpha,\om}:=\left\{f:\, \int_{\C}|f(z)|^pe^{-p\frac{\a}{2}|z|^2}\,\om(z)dA(z)<\infty \right\}. $$
	Using these equivalent norms for $\mathcal{F}^q_{\beta,\omega}$  we
	characterize the  Carleson measures for weighted Fock-Sobolev spaces $\mathcal{F}^{q,n}_{\beta,\om}$.
\end{abstract}
\maketitle

\section{Introduction}

\par Let $\C$ be  the complex plane
and denote by $H(\C)$ the space of entire functions.  For
$0<p,\a<\infty$, let $\Lpa$ be  the space of measurable functions such that
$$ \|f\|^p_{\Lpa}:=\frac{p\a}{2\pi}\int_{\C}|f(z)|^pe^{-p\frac{\a}{2}|z|^2}\,dA(z)<\infty,$$
where $dA$ denotes the Lebesgue measure in $\C$. The
classical Fock
spaces  $\Fpa$ consists of the $f\in H(\C)$ such that  $f\in\Lpa$ \cite{tung,ZhuFock}. A function $\om:\C\to [0,\infty)$ is a weight if $\om$ belongs
to $L^1_{loc}(\C,dA)$. In this paper, we will deal
  with the weighted  spaces
of measurable functions,
$$\Lpaom:=\left\{f:\, \|f\|^p_{\Lpaom}=\int_{\C}|f(z)|^pe^{-p\frac{\a}{2}|z|^2}\,\om(z)dA(z)<\infty \right\}, $$
and the weighted Fock spaces $\Fpaom=\Lpaom\cap H(\C)$.
\par The
  orthogonal projection  $P_\a:\mathcal{L}^2_\a\to \mathcal{F}^2_\a$
 coincides with the integral operator
\begin{equation*}
P_\a(f)(z):= \frac{\a}{\pi}\int_{\C}
f(\z)e^{\a\overline{\z}z}e^{-\a|\z|^2}\,dA(\z), \quad f\in
\mathcal{L}^2_\a.
\end{equation*}
The boundedness of projections on $L^p$-spaces   is a classical and interesting  topic,
with a flurry of activity in the recent years,
  which has plenty of applications on
operator theory \cite{AlCo,BB,IsraJOT2014,PelRatproj,Zhu,ZhuFock}. It is known that $P_\a$, and
the positive linear operator
\begin{equation*}\begin{split}
P^+_\a(f)(z) &:= \frac{\a}{\pi}\int_{\C}
f(\z)|e^{\a\overline{\z}z}|e^{-\a|\z|^2}\,dA(\z),
\quad
f\in \mathcal{L}^2_\a,
 \end{split}\end{equation*}
are  bounded on the classical $L^p(\C,
e^{-\a|\cdot|^2}\,dA)$ if and only if $p=2$.
However they are bounded  on
$\Lpa$ for any $p\ge 1$, \cite[Theorem~2.20]{ZhuFock} (see also
\cite{tung}).
Recently, the weights such that $P_\alpha$ is bounded on $\Lpaom$ have been
described in
\cite[Theorem~3.1]{IsraJOT2014}.
In order to state this result, we need  a bit more of notation.
 If $E\subset\C$ is measurable, we denote by
$|E|$  its Lebesgue area measure and let
$\om(E):=\int_{E}\om(z)\,dA(z)$.  Throughout the paper
$Q$  denotes a square in $\R^2$.
We write $l(Q)$ for its
side length and
as usual,  $\frac{1}{p}+\frac{1}{p'}=1$, for $1\le p\le \infty$.

Let $A_{p,r}$ be the class of weights $\om$ on $\C$ such
that $\om(z)>0$ a.e. on $\C$ and
\begin{equation*}
\mathcal{C}_{p,r}(\om):=\sup_{Q,\,l(Q)=r}\left( \frac{1}{|Q|} \int_{Q}
\om\,dA\right)\left(  \frac{1}{|Q|}\int_{Q} \om^{-\frac{p'}{p}}\,dA
\right)^{\frac{p}{p'}}<\infty.
\end{equation*}
If $z_0\in\C$
and $\g\in\mathbb{R}$, $\om(z)=(1+|z+z_0|)^\g$  and
$\om(z)=e^{\g|z+z_0|}$ belong to $\cup_{p> 1}\Ap$ (see the proof of Lemma~\ref{le:weightdistor} below).

 We recall that for a measurable function $f$ in $\C$ the Berezin transform of $f$ is given by
$$\widetilde{f}^{(\a)}(z):=\frac{\a}{\pi}\int_{\C}e^{-\alpha|z-u|^2}f(u)\,dA(u).$$
\begin{lettertheorem}\label{th:isra}
Let $\beta>0$, $1<p<\infty$ and $\om$ a weight  such
that $\om(z)>0$ a. e. on $\C$. Then, the
following conditions are equivalent;
\begin{enumerate}
\item \label{item:isra1} $\om$ belongs to $A_{p,r}$ for some $r>0$;
\item \label{item:isra2} $\om$ belongs to $A_{p,r}$ for any  $r>0$;
\item \label{item:isra3} For each
$\a, \g>0$
\begin{equation*}
\sup_{z\in \C} \widetilde{\om}^{(\a)}(z)\left(
\widetilde{\om^{-\frac{p'}{p}}}^{(\g)}(z)\right)^{\frac{p}{p'}}<\infty;
\end{equation*}
\item \label{item:isra4} $P_\beta$ is bounded on $\mathcal{L}^p_{\beta,\om}$;
\item \label{item:isra5} $P^+_\beta$ is bounded on $\mathcal{L}^p_{\beta,\om}$.
\end{enumerate}
\end{lettertheorem}

The conditions \eqref{item:isra3} and \eqref{item:isra5} in Theorem~\ref{th:isra} do not appear
in the statement of \cite[Theorem~3.1]{IsraJOT2014}. The equivalence
of \eqref{item:isra5} with \eqref{item:isra2} follows easily from the proof of \cite[Theorem~3.1]{IsraJOT2014}. Moreover, the \lq\lq
invariant\rq\rq description of $A_{p,r}$ weights  obtained via the
 Berezin transform was proved
 in \cite[p.~397]{IsraJOT2014} (it also
follows from Lemma~\ref{le:2} below).

\par In order to complete the picture about the boundedness of the Bergman projection on
$\mathcal{L}^p_{\beta,\om}$, we will deal with the remaining case
$p=1$.
 In particular, fixed $r>0$, we say that a weight $\om\in
A_{1,r}$ if $\om(z)>0$ a.e. on $\C$ and there is a constant $C=C(r,\om)$ such that for any
square with $l(Q)=r$,
\begin{equation*}
\frac{\om(Q)}{|Q|}\le C \om(u),\quad\text{a. e.
$u\in Q$.}
\end{equation*}
We denote
\begin{equation*}
\mathcal{C}_{1,r}(\om):=\sup_{Q,\,l(Q)=r}\frac{\om(Q)}{|Q|\text{infess}_{u\in
Q}\om(u)}.
\end{equation*}

We will prove an analogous  version  of Theorem \ref{th:isra} for $A_{1,r}$ weights (see Proposition \ref{pr:p=1} below).

In view of Theorem~\ref{th:isra} and Proposition~\ref{pr:p=1},
  the condition  $A_{p,r}$ does
not depend on $r$. So from now on, following the notation in \cite{IsraJOT2014}, we will write $\Ap$ for these
classes of weights.
 Moreover, $\om\in \Ap$ if and only if  $P^+_\a$ is
bounded on $\Lpaom$ for some (equivalently for any) $\a>0$.

Similarly to Muckenhoupt weights, we denote
$$
\Ainfty:= \bigcup_{ 1\le q<\infty}\Aq .
$$
 A  discussion about properties and examples of  $\Ap$-weights,
 as well as similarities and differences with the classical Muckenhoupt weights,
  will be provided in Section \ref{sec:apv}.
 In particular,  we  sketch the proof of a
 characterization of $\Ainfty$-weights, in the sense of Kerman and Torchinsky \cite{KT}, obtained in \cite{DMOMathZ16}.

Equipping a space of functions with different equivalent norms
results to be quite effective in the study of function and concrete operator theory \cite{AlCo,ConsPelJGA2016,PelRatproj,Zhu,ZhuFock}. In the context of Fock spaces,
 it is known that (see \cite{ChoZhu2012,ConsPAMS2012,ConsPelJGA2016}) for
any $\a,p>0$
\begin{equation}\label{eq:LPfcla}
\| f\|^p_{\Fpa}\asymp \sum_{j=0}^{k-1}|f^{(j)}(0)|^p+
\int_{\C}\frac{|f^{(k)}(z)|^p}{(1+|z|)^{kp}}e^{-p\frac{\a}{2}|z|^2}\,dA(z),\quad
f\in H(\C),
\end{equation}
that is, the distortion function in the above Littlewood-Paley
formula is $\psi(z)=\frac{1}{1+|z|}$. We will use Theorem~\ref{th:isra}, among other techniques,
to prove
an extension of \eqref{eq:LPfcla} to weighted
Fock spaces
induced by $\Ainfty$~-~weights.

\begin{theorem} \label{th:LPformula}
 Let $\om\in \Ainfty$. Then, for $0<p<\infty$, $k\in\N$ and
$\a>0$
\begin{equation*}
\| f\|^p_{\Fpaom}\asymp
 \sum_{j=0}^{k-1}|f^{(j)}(0)|^p+
\int_{\C}|f^{(k)}(z)|^pe^{-p\frac{\a}{2}|z|^2}\frac{\om(z)}{(1+|z|)^{kp}}\,dA(z),\quad
f\in H(\C).
\end{equation*}
\end{theorem}

As for the proof of Theorem~\ref{th:LPformula}, it is enough to prove the case $k=1$ because
$$\om_{\gamma}(z):= \frac{\om(z)}{(1+|z|)^{\gamma}}\in \Ainfty,\qquad \g\in\mathbb{R},$$
  whenever $\om\in\Ainfty$ (see Lemma \ref{le:weightdistor} below).

\par Later on, we apply Theorem~\ref{th:LPformula} to study the
boundedness of the differentiation and integration operator. We denote
 $D^{(n)}(f)=f^{(n)}$, \, $n\in \N$,
$D^{(0)}(f)=f$, $D^{(-1)}(f)(z)=\int_0^z f(\z)\,d\z$ and
$D^{(-n)}(f)=D^{(-n+1)}\circ D^{(-1)}(f)$,\, $n\in\N$.
 Given $0<\a, p, q<\infty$,  a positive Borel measure
$\mu$ on $\C$ is a $q$-Carleson measure for $\Fpaom$ if the identity
$I_d: \Fpaom\to L^q(\mu)$ is a bounded operator.  We  denote by $D(z,r)$
the Euclidean disc of center $z$ and radius $r>0$.

\begin{theorem}\label{th:differentiation}
 Let
 $\a\in (0,\infty)$, $n\in\Z$, $\om\in \Ainfty$ and
$\mu$ a  finite positive Borel measure on $\C$ such that $\sup_{k\in \N}\| z^k\|_{L^q(\mu)}<\infty$.
\par If $0<p\le q<\infty$, the following conditions are equivalent:
\begin{enumerate}
\item $D^{(n)}\,:\Fpaom\to L^q(\mu)$ is a bounded operator;
\item $\mu$ is a $q$-Carleson measure for $
\mathcal{F}^p_{\alpha, \om_{np}}$;
\item The function
$$
G(a):=\frac{1}{\left(\om_{np}(D(a,1))\right)^{\frac{q}{p}}}\int_{D(a,1)}
	e^{\a\frac{q}{2}|z|^2}\,d\mu(z)
$$
is bounded on $\C$.
\end{enumerate}
Moreover,   if $n\in\N\cup\{0\}$
\begin{equation}\label{normidentityfs}
\|I_d\|_{\mathcal{F}^p_{\alpha, \om_{np}}\to L^q(\mu)}^q\asymp
\|D^{(n)}\|_{\Fpaom\to L^q(\mu)}^q \asymp \|G\|_{L^\infty},
\end{equation}
and for $n$  a negative integer,
\begin{equation}\label{normidentityfsnegative}
\|I_d\|_{\mathcal{F}^p_{\alpha, \om_{np}}\to L^q(\mu)}^q +C_{\mu,n}\asymp
\|D^{(n)}\|_{\Fpaom\to L^q(\mu)}^q+ C_{\mu,n} \asymp \|G\|_{L^\infty}+C_{\mu,n}.
\end{equation}
where $C_{\mu,n}=\max_{k=0,\dots, -n-1}\|z^k\|^q_{L^q(\mu)}<\infty$.
\par If $0<q<p<\infty$, the following conditions are equivalent:

\begin{enumerate}
\item $D^{(n)}\,:\Fpaom\to L^q(\mu)$ is a bounded operator;
\item $\mu$ is a $q$-Carleson measure for $
\mathcal{F}^p_{\alpha, \om_{np}}$;
\item The function
$$H(u):= \frac{1}{\om_{np}\left(D(u,1)\right)}\int_{D(u,1)}
e^{\frac{q\a |z|^2}{2}}\,d\mu(z)
$$
is  in
$L^{\frac{p}{p-q}}(\mathbb{C},\om_{np}).$
\end{enumerate}

Moreover, if $n\in\N\cup\{0\}$,

\begin{equation}\label{normidentityfs2}
\|I_d\|_{\mathcal{F}^p_{\alpha, \om_{np}}\to L^q(\mu)}^q \asymp
\|D^{(n)}\|_{\Fpaom\to L^q(\mu)}^q
 \asymp \|H\|_{L^{\frac{p}{p-q}}(\C,\om_{np})},
\end{equation}
and if $n$ is a negative integer,
\begin{equation}\label{normidentityfs2negative}
\|I_d\|_{\mathcal{F}^p_{\alpha, \om_{np}}\to L^q(\mu)}^q + C_{\mu,n} \asymp
\|D^{(n)}\|_{\Fpaom\to L^q(\mu)}^q +C_{\mu,n}
 \asymp \|H\|_{L^{\frac{p}{p-q}}(\C,\om_{np})}+ C_{\mu,n}.
\end{equation}
\end{theorem}
\par It is worth noticing that the  technical assumption  $\sup_{k\in \N}\| z^k\|_{L^q(\mu)}<\infty$ in Theorem~\ref{th:differentiation},
which is only used when $n$ is a negative integer,
is not a real restriction (see Lemma~\ref{le:2} below).
Let us also consider the Fock-Sobolev spaces $\mathcal{F}^{p,n}_{\a,\om}$,
 $n\in\N$,
  of the entire functions such that
$$
\|f\|^p_{\mathcal{F}^{p,n}_{\a,\om}}:=
\sum_{k=0}^{n-1}|f^{(k)}(0)|^p+
\int_{\C}|f^{(n)}(z)|^pe^{-p\frac{\a}{2}|z|^2}\,\om(z)dA(z)<\infty.
$$
 By Theorem \ref{th:LPformula} and Lemma~\ref{le:weightdistor} (below),  $\mathcal{F}^{p,n}_{\a,\om}$ coincides with $\mathcal{F}^{p}_{\a,\om_{-np}}$, so
Theorem~\ref{th:differentiation} also gives a  description of the $q$-Carleson
measures for  Fock-Sobolev spaces induced by $\Ainfty$-weights.
A characterization of $q$-Carleson measures  for classical Fock-Sobolev spaces is provided in \cite{MeCAOT14}, see also \cite{ChoZhu2012, Isralowitz-Zhu}.

As a consequence of Theorem \ref{th:differentiation} we obtain characterizations of pointwise multipliers between weighted Fock spaces,
$\Fpaom$ and $\mathcal{F}^q_{\b,\eta}$, for $0<q,p<\infty$, $\eta$ a weight and $\om\in\Ainfty$, which are particularly neat for
$\om=\eta$,   see Theorem~\ref{thm:PMom} below.

This paper is organized as follows. In Section \ref{sec:apv}
we study $\Ap$-weights and prove a collection of preliminary results which will be employed to
prove  Proposition~\ref{pr:p=1}. These estimates will  also be used to prove  Theorem \ref{th:LPformula}
in  Section \ref{sec:LP}.
 We prove Theorem \ref{th:differentiation} in Section \ref{sec:Carleson}.
Finally, in Section \ref{sec:PM-E} we apply the descriptions on Carleson measures for $\Fpaom$ to obtain
 characterizations of  pointwise multipliers between weighted Fock spaces.

For two real-valued functions $E_1,E_2$ we write $E_1\asymp E_2$, or
$E_1\lesssim E_2$, if there exists a positive constant $k$
independent of the argument such that $\frac{1}{k} E_1\leq E_2\leq k
E_1$, respectively $E_1\leq k E_2$.
Given a weight $\om\in \Ap$, $1<p<\infty$, we denote the  conjugate weight $\om^{-p'/p}$ by $\om'$.
We write $\mathcal{R}(z)$ for the
real part of a complex number $z$.

\section{$\Ap$ weights versus classical
	 Muckenhoupt $A_p$ weights.}\label{sec:apv}

In this section we briefly discuss some aspects of the theory of $\Ap$-weights, in particular we
 give examples  and descriptions of these classes of weights. Some of these characterizations are not used in the proofs
 of further results in this work. However we consider relevant to understand the basic features of $\Ap$-weights,
specially comparing them with their analogues into the theory of classical Muckenhoupt weights.
We recall that $\om$ belongs to the class of Muckenhoupt weights
$A_p$, $1<p<\infty$, if
\begin{equation*}
A_p(\om):= \sup_{Q} \frac{\om(Q)}{|Q|} \left(
\frac{\om'(Q)}{|Q|}
\right)^{\frac{p}{p'}}<\infty,
\end{equation*}
where the supremum runs over all the squares $Q\subset \C$.
In addition, $\om\in A_1$ if \begin{equation*}
A_1(\om):=
\sup_{Q} \frac{\om(Q)}{|Q|\text{infess}_{u\in Q} \om(u)} <\infty.
\end{equation*}
An extensive  study  of these weights can be found, for instance, in  \cite{GarciaCuerva,SteinHarmonic93}.

\subsection{The class $\Ap$ with $1\le p<\infty$}\quad\par

 We begin gathering some 
 properties of  $\Ap$-weights
 which are analogous to the properties of  Muckenhoupt weights (see \cite[p.~195]{SteinHarmonic93}). We write $Q_r(z)$
 for the square of center $z\in\C$, with sides parallel to the coordinate axes and $l(Q)=r$.

 \begin{lemma}\label{le:weightdistor}
 	Let $p\ge 1$, $\g\in\mathbb{R}$ and  $\om\in \Ap$. Then
 	$$\om_\g=\frac{\om(z)}{(1+|z|)^\g}\in\Ap.$$
 \end{lemma}
 \begin{proof}
 	Fix $r>0$. Then for any  $z_0\in\C$,
 	$$ \frac{(1+|z|)}{1+r}\le (1+|z_0|)\le  (1+r)(1+|z|), \quad z\in
 	Q_r(z_0).$$
 So, it follows
 	that
 	$\frac{\om(z)}{(1+|z|)^\g}\in \Ap$.
 \end{proof}

\begin{proposition} \label{prop:aprest}
	Let $\om$ be a weight  such
that $\om(z)>0$ a. e. on $\C$. Then,
\begin{enumerate}
\item\label{item:aprest1}  If $1<p<\infty$, $\om\in \Ap$  if and only if
$\om'=\om^{-p'/p}\in \Apprime$.
\item\label{item:aprest3}   If $1\leq p<\infty$,  $\om\in \Ap$ if and only if for some (equivalently for any) $r>0$ there is a constant
$C_r>0$ such that for any nonnegative measurable function $f$
\begin{equation}\label{eq:meanapdesc}
\left(\frac{1}{|Q|}\int_{Q} f dA\right)^p\le
\frac{C_r}{\om(Q)}\int_{Q} f^p\om dA,\quad\text{for any $Q$ with
$l(Q)=r$. }
\end{equation}
\item\label{item:aprest2} If $1\leq p< q<\infty$, $\Ap\subsetneq \Aq$.
\item\label{item:aprest4} If $1\leq p<\infty$, $A_p\subsetneq \Ap$.

\end{enumerate}
\end{proposition}
\begin{proof}
\eqref{item:aprest1} follows from the definition of  $\Ap$-weights.
Let us prove \eqref{item:aprest3}. If $p>1$,  H\"older's inequality  gives that any $\om\in \Ap$ satisfies \eqref{eq:meanapdesc}. Conversely  if  a
weight satisfies \eqref{eq:meanapdesc}, by choosing
$f=(\om+\ep)^{-p'/p}$, $\ep>0$, and taking limit $\ep\to
0^+$, we deduce that $\om\in\Ap$. Now,
if $\om\in A_1^{restricted}$, then for any $f \geq 0$, and any square $Q$,
\begin{equation}\label{eqn:a1restricted}
\frac{1}{|Q|}\int_{Q} f dA\leq \frac{C_r}{ \om(Q)}\int_Q f\om dA.
\end{equation}
On the other hand, if $\om$ satisfies \eqref{eqn:a1restricted} then H\"older's inequality gives that for any $1<p<\infty$,
\begin{equation}\label{eqn:Ap}\left(\frac{1}{|Q|}\int_{Q} f dA\right)^p\leq \frac{C_r^p}{ \om(Q)}\int_Q f^p\om dA,
\end{equation}
which implies that $\om\in \Ap$ and $\mathcal{C}_{p,r}(\om)\leq C_r^p$.
Since
$$\lim_{p\to 1^+}\left(  \frac{1}{|Q|}\int_{Q} \om^{-\frac{p'}{p}}\,dA
\right)^{\frac{p}{p'}}=
\text{supess}_{u\in
Q}\om^{-1}(u),
$$
  we conclude that $\om\in A_1^{restricted}$.

 \eqref{item:aprest2} for $p>1$ is a consequence of H\"older's inequality, where  $\om(z)=|z|^{2(\delta-1)}$, for $\delta\in (p,q)$ proves that the inclusion is strict.
If  $\om\in A_1^{restricted}$, the assertion follows from the fact that  \eqref{eqn:Ap} holds for any $1<p<\infty$.

Finally, it is
clear that $A_p\subset \Ap$. In order to see that the embedding
is strict, one can choose the weight $\om(z)=(1+|z|^2)^{p-1}$, for
$p>1$. In fact,
 Lemma~\ref{le:weightdistor}
 shows that $\om\in\Ap$ and since
\begin{equation}\begin{split}\label{eq:ex1}
&\frac{1}{|D(0,r)|} \int_{D(0,r)} \om\,dA\left(
\frac{1}{|D(0,r)|}\int_{D(0,r)} \om^{-p'/p}\,dA
\right)^{\frac{p}{p'}}
\\ &
\asymp \left(\log\frac{1}{1+r^2}\right)^{p-1}\to \infty, \quad r\to \infty,
\end{split}\end{equation} $\om$ is not a classical $A_p$-Muckenhoupt
weight.
 As for $p=1$, $\om(z)=e^{|z|}\in A^{restricted}_1\setminus A_1$.
\end{proof}

The  following lemma is proved in
\cite[Lemma~3.4]{IsraJOT2014}.

\begin{letterlemma}\label{le:isra34}
	For
each $r > 0$, let $r\mathbb{Z}^2$ denote the set $\{rk_1 +i rk_2:\,
k_1,k_2\in \mathbb{Z}\}$. If  $\om\in
	\Ap$, then there is a constant $M=M(\om,r,p)$ such that
	$$\frac{\om (Q_{r}(\nu))}{\om (Q_{r}(\nu'))}\le M^{|\nu-\nu'|}$$
	for all $\nu,\nu'\in r\mathbb{Z}^2$.

\end{letterlemma}

 \begin{remark}\label{rem:discoscomparables}
 {\rm
The above lemma shows that  for each $L>0$, $\om(Q_r(\nu))\asymp \om(Q_r(\nu'))$ if  $|\nu-\nu'|<L$ and $\om\in\Ainfty$.
So, fixed $N>r$ we have that $$\om(Q_r(\nu))\asymp \om(Q_{Nr}(\nu')).$$
 Combining these results it follows that
 $$
 \om(Q_r(z))\asymp \om(Q_R(w)),\quad\text{for any $z,\,w$ such that $|z-w|<L$,}
 $$
whenever $\om\in\Ainfty$, $r<L$ and $R\in(r,Nr)$.

 In particular, if $\om\in\Ainfty$, for any $t>0$ and $N\in\N$ there is
 $C=C(\om,N,p)$ such that
 \begin{equation}\label{eq:discoscomparables}
 \om\left( D(a,t)\right)\le  \om\left( D(a,Nt)\right)\le C \om\left(
 D(a,t)\right),\quad a\in\C.
 \end{equation}
 \par
 As a consequence, squares of fixed length can be replaced by
discs of  fixed radius in the study of $\Ap$ weights.
}
\end{remark}
Now, let us observe that the class $\Ap$  is invariant under
translations.
\begin{lemma}\label{le:translation}
	Let $p\in [1,\infty)$. Then, a weight $\om$  belongs to $\Ap$ if and
	only for any (equivalently for some) $a\in\C$, $\om_{[a]}(u):=\om(a+u)\in \Ap$.
	Moreover,
	$$\mathcal{C}_{p,r}(\om)=\mathcal{C}_{p,r}(\om_{[a]}).$$
\end{lemma}
\begin{proof}
	The proof follows immediately from the invariance under translations of the Lebesgue measure and the definition of the $\Ap$-weights.
	
\end{proof}

 A fundamental property of the classical Muckenhoupt weights is the reverse H\"older inequality \cite[Theorem~7.4]{Duandilibro}.
 So, it is natural to ask whether or not  for each  $\Ap$-weight $\om$
 there exists $\ep>0$ and  $C(\ep,r)>0$ such that
\begin{equation}\label{eq:reverse}
\left(\frac{1}{|Q|}\int_{Q}
\om^{1+\ep}(\z)\,dA(\z)\right)^{\frac{1}{1+\ep}} \ \le C(\ep,r)
\frac{1}{|Q|}\int_{Q} \om(\z)\,dA(\z),
\end{equation}
for every cube $Q$ with $l(Q)=r$.
The following considerations provide a negative answer to this question.

\begin{remark}

\begin{itemize}
\item
The weight
  $\om(z)=|z|^{-2}\left(\log \frac{1}{|z|}\right)^{-2}
\in A^{restricted}_2$,  but for each $\ep>0$, $\om^{1+\ep}\notin A_2^{restricted}$. So $\om$ does not satisfy  \eqref{eq:reverse}.
\item
The weight
$$
v(z)=|z|^{2}\left(\log \frac{1}{|z|}\right)^{2}$$ is in $
A^{restricted}_2\setminus \cup_{1<p<2}\Ap
.$
 That is, the $\Ap$-weights
do not satisfy the natural analogue of the $(p-\e)$-condition which
holds for the classical $A_p$-Muckenhoupt weights. Moreover, sincee
\begin{equation*}\begin{split}
&\frac{1}{|D(0,r)|^2}\left( \int_{D(0,r)} v(z)\,dA(z) \right)\left(\int_{D(0,r)}
v^{-1}(z)\,dA(z)\right) \\ & \asymp
\log\frac{1}{r}\to \infty, \quad r\to 0,
\end{split}\end{equation*}
we deduce that for any  $r_0>0$, the $\Ap$-weights
do not
satisfy  the uniform condition
\begin{equation*}
\sup_{Q_r, r\le r_0} \frac{1}{r^{2p}} \int_{Q_r}
\om\,dA\left( \int_{Q_r} \om^{-\frac{p'}{p}}\,dA
\right)^{\frac{p}{p'}}<\infty.
\end{equation*}
\end{itemize}
\end{remark}

\subsection{The class $\Ainfty$}\quad\par

\par In the classical setting, there are a good number of equivalent conditions which describe
the class  $A_\infty=\bigcup_{1\le q<\infty} A_q$
 (see \cite{DMOMathZ16}, \cite[Chapter~5]{SteinHarmonic93} or
\cite[p.~149]{Duandilibro}). So, it is natural to ask
whether or not the class $\bigcup_{q\ge 1}\Aq$ can be described by
neat analogous conditions to those describing the class $A_\infty$.
With this aim, we introduce the following class of weights.

\begin{definition}
 A weight $\om$ satisfies the
$KT_r$-property, $r>0$, if there exist constants $r>0$, $\delta\in (0,1)$ and
$C_r>0$ such that for any square $Q$ with
 $l(Q)=r$ and every
measurable set $E\subset Q$ it holds that
\begin{equation}\label{eq:KTr}
\frac{|E|}{|Q|}\le C_r\left( \frac{\om(E)}{\om(Q)}\right)^{\delta}.
\end{equation}
\end{definition}

If we replace in \eqref{eq:KTr} the constant $C_r$ by an absolute constant $C$
and $Q$ runs over all the squares $Q$, it is obtained a  condition
which describes
the class $A_\infty$ of
Muckenhoupt weights \cite[Theorem~3.1]{DMOMathZ16}. It was
introduced by Kerman and Torchinsky in \cite[Proposition~1]{KT} in
order to describe the restricted weak-type for the Hardy-Littlewood
maximal operator. It follows from \cite[Theorem~3.1]{DMOMathZ16} (a
result which holds for general basis) that, for each $r>0$, a weight
$\om\in \Ainfty$  if and only if it satisfies
the $KT_r$-property.
For the sake of completeness  we offer a
direct short proof  based on the ideas
of the proof of \cite[Thorem~3.1]{DMOMathZ16}.

\begin{letterproposition}\label{pr:DKS}
Let $\om$ be a weight. Then, the following conditions are
equivalent:
\begin{enumerate}
\item \label{item:DKS1} $\om \in \Ainfty$;
\item \label{item:DKS2} $\om$ satisfies the Kerman-Torchinsky
$KT_{r}$-property for any $r>0$;
\item \label{item:DKS3} $\om$ satisfies the Kerman-Torchinsky
$KT_{r}$-property for some $r>0$.
\end{enumerate}
\end{letterproposition}
\begin{proof}

\par \eqref{item:DKS1} $\Rightarrow$ \eqref{item:DKS2}. Fix $r>0$ and take $p>1$ such that $\om\in
A_{p,r}$. Then, if 
 $l(Q)=r$ and $E\subset Q$,
then
\begin{equation*}\begin{split}
|E| &\le
\om(E)^{1/p}\left(\om'(E)\right)^{1/p'} \le
\om(E)^{1/p}\left(\om'(Q)\right)^{1/p'}
\\ & \le
\mathcal{C}^{1/p}_{p,r}(\om)\om(E)^{1/p}|Q|^{1/p'}\left(\frac{|Q|}{\om(Q)}\right)^{1/p},
\end{split}\end{equation*}
that is
$\frac{|E|}{|Q|}\le\mathcal{C}^{1/p}_{p,r}(\om)\left(\frac{\om(E)}{\om(Q)}\right)^{1/p},$
which gives \eqref{item:DKS2} with $\delta=\frac{1}{p}$.
\par \eqref{item:DKS3} $\Rightarrow$ \eqref{item:DKS1}. Let $Q$ be a square with $l(Q)=r$.
For each $\lambda>0$, let us denote $E_\lambda=\{z\in Q:
\om(z)<\frac{1}{\lambda}\}$. Then, by hypothesis there is $\delta\in
(0,1)$ and $C_r>0$
$$
\lambda\om(E_\lambda)\le |E_\lambda|\le
C_r|Q|\left(\frac{\om(E_\lambda)}{\om(Q)}\right)^{\delta},$$ that is
\begin{equation*}
\lambda^{\frac{1}{1-\delta}}\om(E_\lambda)\le
C_{r,\delta}\left(\frac{|Q|}{\om(Q)^\delta}\right)^{\frac{1}{1-\delta}}.
\end{equation*}
Take $p$ such that $p'\in \left(1, \frac{1}{1-\delta}\right)$. Then,
if we denote $d\om=\om dA$, for any $M>0$
\begin{equation*}\begin{split}
\int_{Q}\om^{-p'/p}\,dA & =
\int_{Q}\om^{-p'}\,d\om=p'\int_{0}^\infty
\lambda^{p'-1}\om(E_\lambda)\,d\lambda \\ & \le \om(Q) p'\int_{0}^M
\lambda^{p'-1}\,d\lambda+C_{r,\delta}\left(\frac{|Q|}{\om(Q)^\delta}\right)^{\frac{1}{1-\delta}}\int_{M}^\infty
\lambda^{p'-1-\frac{1}{1-\delta}}\,d\lambda
\\ &=\om(Q)
M^{p'}+C_{r,p,\delta}\left(\frac{|Q|}{\om(Q)^\delta}\right)^{\frac{1}{1-\delta}}M^{p'-\frac{1}{1-\delta}},
\end{split}\end{equation*}
so choosing $M=\frac{|Q|}{\om(Q)}$ we get
$$\frac{\int_{Q}\om^{-p'+1}\,dA}{|Q|}\le
C_{r,p,\delta}\left(\frac{|Q|}{\om(Q)}\right)^{p'-1},$$ that is
$\om\in A_{p,r}$, which together with Theorem~\ref{th:isra} finishes
the proof.
\end{proof}

\subsection{The class $A_1^{restricted}$}
The primary  aim of this section consists on proving the following result.
\begin{proposition}\label{pr:p=1}
Let $\beta>0$ and $\om$ be a weight  such
that $\om(z)>0$ a.e. on $\C$. Then, the following
conditions are equivalent;
\begin{enumerate}
\item \label{item:p=11} $\om$ belongs to $A_{1,r}$ for some $r>0$;
\item \label{item:p=12} $\om$ belongs to $A_{1,r}$ for any  $r>0$;
\item \label{item:p=13} For any $\alpha>0$ there is a positive constant $C$ such
that
$$\widetilde{\om}^{(\a)}(z)\le C \om(z)\quad\text{a.e. $z\in\C$;}$$
\item \label{item:p=14} $P_\beta$ is bounded on $\mathcal{L}^1_{\beta,\om}$;
\item \label{item:p=15} $P^+_\beta$ is bounded on $\mathcal{L}^1_{\beta,\om}$.
\end{enumerate}
\end{proposition}
Some preliminary results, which will also be  used to prove Theorem~\ref{th:LPformula}, are needed.

\begin{lemma}\label{le:2}
Assume that  $\om\in \Ainfty$. Then, there exists $r_0$ such that for
any $r\in (0,r_0)$ and for any $\b>0$
$$\int_{\C} e^{-\beta|z|^2}\om(z)\,dA(z)\le C(\om,r,\beta)\om\left(Q_r(0)\right)<\infty.$$
\end{lemma}

\begin{proof}
 Choose $r_0>0$ such that $|z-\nu|\le 1$ if $z\in Q_{r_0}(\nu)$. So
\begin{equation}\label{eq:4}
|\nu|\le 1+|z|,\quad z\in Q_r(\nu),\,r\in (0,r_0).
\end{equation}
 Given
$\b>0$, choose $\a\in (0,\b)$ and $R=R(\beta)$ such that
$$\a|z|^2+2\a|z|\le \b|z|^2,\quad \text{if $|z|\ge R$}.$$
By \eqref{eq:discoscomparables}
it is enough to prove that $$ \int_{\C}
e^{-\a|z|^2-2\a|z|}\om(z)\,dA(z)\le
C(\a,r,\om)\om\left(Q_r(0)\right)<\infty. $$ Now, bearing in mind
\eqref{eq:4} and Lemma~\ref{le:isra34} we deduce that
\begin{equation*}\begin{split}
\int_{\C} e^{-\a|z|^2-2\a|z|}\om(z)\,dA(z) & = \sum_{\nu\in
r\mathbb{Z}^2}\int_{Q_r(\nu)} e^{-\a|z|^2-2\a|z|}\om(z)\,dA(z)
\\ & \le e^{\a}  \sum_{\nu\in
r\mathbb{Z}^2} e^{-\a|\nu|^2}\int_{Q_r(\nu)} \om(z)\,dA(z)
\\ & \le  C(\a,r,\om)\om\left(Q_r(0) \right)\sum_{\nu\in
r\mathbb{Z}^2} e^{-\a|\nu|^2}M^{|\nu|}
\\ & \le  C(\a,r,\om)\om\left(Q_r(0)\right)<\infty.
\end{split}\end{equation*}
 This finishes the proof.
\end{proof}

\begin{lemma}\label{le:3}
If $\om\in \Ainfty$, then for any $\b,t,r>0$ and $\gamma\in\mathbb{R}$
$$\int_{\C} e^{-\beta|u|^2}\om\left(D(u,t) \right)^\gamma \,dA(u)
\le C(\g,\b,r,\om,t)\om\left(Q_{r}(0)\right)^\gamma <\infty.$$
\end{lemma}

\begin{proof}
Fixed $\beta>0$, choose $r_0$ and $\alpha$ as in the proof of Lemma~\ref{le:2}.
 It is enough to prove that $$ \int_{\C}
e^{-\a |u|^2-2\a |u|}\om\left(D(u,t) \right)^\gamma\,dA(u)\le
C(\a,r,\om,\g,t)\om\left(Q_{r}(0)\right)^\gamma<\infty.$$ Next, take
$N=N(t,r)\in \N$ such that
\begin{equation}\label{eq:inclusion}
D(u,t)\subset Q_{N}(\nu),\quad u\in Q_r(\nu),\quad  \nu\in
r\mathbb{Z}^2.
\end{equation}
Now, bearing in mind \eqref{eq:4}, \eqref{eq:inclusion},
Lemma~\ref{le:isra34} and
\eqref{eq:discoscomparables}, we deduce that
\begin{equation*}\begin{split}
\int_{\C} e^{-\a|u|^2-2\a|u|}&\om\left(D(u,t) \right)^\g\,dA(u)
\\ & \le e^{\a}  \sum_{\nu\in
r\mathbb{Z}^2} e^{-\a|\nu|^2}\int_{Q_r(\nu)} \om\left(D(u,t)
\right)^\g \,dA(u)
\\ & \le r^2e^{\a}  \sum_{\nu\in
r\mathbb{Z}^2} e^{-\a|\nu|^2} \om\left(Q_{N}(\nu)\right)^\g
\\ & \le  C(\a,r)\sum_{\nu\in
r\mathbb{Z}^2} e^{-\a|\nu|^2}M^{\g|\nu|}\om\left(Q_{N}(0)\right)^\g
\\ & \le  C(\a,r,\om,t)\om\left(Q_{r}(0)\right)^\g\sum_{\nu\in
r\mathbb{Z}^2} e^{-\a|\nu|^2}M^{\g|\nu|}
\\ & \le  C(\a,r,\om,\g,t)\om\left(Q_{r}(0)\right)^\g
<\infty.
\end{split}\end{equation*}
Next assume that $\g<0$. Take $p\in (1,\infty)$ such that
$\om\in\Ap$.
Then,
$\om'\in \Ainfty$ and
$$\om\left(D(u,t) \right)^\g\asymp \om'\left(D(u,t)
\right)^{-p\g/p'}, \quad u\in\C,$$ so the result follows by
applying the above argument to $\om'$. This finishes the proof.
\end{proof}

\subsection{Proof of Proposition~\ref{pr:p=1}.}
\medskip
\par \eqref{item:p=13} $\Rightarrow$ \eqref{item:p=12} . Take $Q$ with $l(Q)=r>0$. Then for any
$z,\z\in Q$, $|z-\z|^2\le 2r^2$, so
\begin{equation*}\begin{split}
e^{-2\a r^2}\om(Q)\le C(\a) \widetilde{\om}^{(\a)}(z)\le
C(\a,\om)\om(z),\quad \text{a.e. $z\in Q$,}
\end{split}\end{equation*}
which implies \eqref{item:p=12} .
\par \eqref{item:p=11} $\Rightarrow$ \eqref{item:p=13}. Fixed $\alpha>0$,   $r>0$ such that
$\om\in A_{1,r}$ and $z\in\C$. Then, by Lemma~\ref{le:translation}
and Lemma~\ref{le:2}
\begin{equation*}\begin{split}
\widetilde{\om}^{(\a)}(z) &=
\int_{\C}e^{-\alpha|\z|^2}\om(\z+z)\,dA(\z)  \le
C(\a,r,\om)\om_{[z]}(Q_r(0))
\\ & = C(\a,r,\om)\om(Q_r(z))
\\ & \le  C(\a,r,\om)\text{infess}_{u\in
Q_r(z)}\om(u)\le \om(z),\quad \text{a.e. $z\in\C$.}
\end{split}\end{equation*}
Therefore, we already have proved
\eqref{item:p=11}  $\Leftrightarrow$ \eqref{item:p=12}  $\Leftrightarrow$ \eqref{item:p=13} .
\par Now, let us observe that $\left( \mathcal{L}^1_{\beta}(\om)\right)^\star\simeq
L^\infty$ via the pairing
$$\langle f,g \rangle_{\mathcal{L}^2_{\beta/2}(\om)}=\int_\C
f(z)\overline{g(z)}e^{-\frac{\b}{2}|z|^2}\om(z)\,dA(z).$$ So,
bearing in mind that the adjoint of $P_\b$ (via the
$\mathcal{L}^2_{\b/2}(\om)$-pairing) is
$$P^\star_\b(f)(z)=\frac{\b}{\pi\om(z)}\int_{\C}f(\z)e^{-\frac{\b}{2}|\z-z|^2+i\b\text{Im}(z\bar{\z})}\om(\z)\,dA(\z),$$
we deduce that \eqref{item:p=14}  holds if and only if
\begin{equation*}
\text{supesss}_{z\in\C}\left|\frac{1}{\om(z)}\int_{\C}f(\z)e^{-\frac{\b}{2}|\z-z|^2+i\b\text{Im}(z\bar{\z})}\om(\z)\,dA(\z)\right|
\le C \| f\|_{L^\infty}.
\end{equation*}
Therefore, it is clear that \eqref{item:p=13} $\Rightarrow$\eqref{item:p=14}. In order to see
the reverse implication, for each $z\in \C$ choose
$f_z(\z)=e^{-i\b\text{Im}(z\bar{\z})}$. The equivalence
\eqref{item:p=13}  $\Leftrightarrow$\eqref{item:p=15}  can be proved in the same way. This finishes the
proof.

\section{A Littlewood-Paley formula for
	$\Fpaom$.}\label{sec:LP}

\subsection{Preliminary results}\quad\par
In this section we will prove a collection of estimates which will be
essential  to prove Theorem~\ref{th:LPformula}.

\begin{lemma}\label{le:1}
	Let $0<p,t<\infty$, $\alpha\ge 0$ and  $\om\in \Ainfty$. Then, there
	exists a constant $C=C(\a,p,\om,t)$ such that
	$$|f(z)|^pe^{-\frac{p\a|z|^2}{2}}\le
	\frac{C}{\om\left(D(z,t)
		\right)}\int_{D(z,t)}|f(u)|^pe^{-\frac{p\a|u|^2}{2}}\om(u)\,dA(u),$$
	for any $z\in\C$ and $f\in H(\C)$.
\end{lemma}

\begin{proof}
	Let be $p_0\in (1,\infty)$ such that $\om\in\Apo$. By subharmonacity for $\a=0$ and by
	\cite[Lemma~1]{OrtegaSeipJAnalMath98} (see also
	\cite[Lemma~7]{ConsPelJGA2016}) for $\alpha>0$, there is a constant $C>0$ such that
	$$|f(z)|^{\frac{p}{p_0}} e^{-\frac{p}{p_0}\frac{\a|z|^2}{2}}\le C \int_{D(z,t)}|f(u)|^{\frac{p}{p_0}}
	e^{-\frac{p}{p_0}\frac{\a|u|^2}{2}}\,dA(u),$$ for any $z\in\C$ and
	$f\in H(\C)$.
	So,
	\begin{equation*}\begin{split}
	&|f(z)|^{\frac{p}{p_0}} e^{-\frac{p}{p_0}\frac{\a|z|^2}{2}}\\
	&\le C
	\left(\int_{D(z,t)}|f(u)|^pe^{-\frac{p\a|u|^2}{2}}\om(u)\,dA(u)
	\right)^{1/p_0} \left(
	\int_{D(z,t)}\om^{-\frac{p'_0}{p_0}}(u)\,dA(u)\right)^{1/p_0'}
	\\ & \le \frac{C}{\left(\om\left(D(z,t) \right)\right)^{1/p_0}}\left(\int_{D(z,t)}|f(u)|^pe^{-\frac{p\a|u|^2}{2}}\om(u)\,dA(u)
	\right)^{1/p_0}.
	\end{split}\end{equation*}
	This finishes the proof.
\end{proof}
The next result follows from Lemma~\ref{le:1}.

\begin{corollary}\label{co:1}
	Let $0<p,\a,t<\infty$ and let be $\om\in \Ainfty$. Then, there
	exists a constant $C=C(\a,p,\om,t)$ such that
	$$|f(z)|\le
	\frac{Ce^{\frac{\a|z|^2}{2}}}{\om\left(D(z,t)
		\right)^{1/p}}\|f\|_{\Fpaom}$$ for any $z\in\C$ and $f\in H(\C)$.
\end{corollary}

Let us recall that $\mathcal{F}^\infty_{\alpha}$, $\a>0$, is the space of the the entire functions $f$ such that
$$\|f\|_{\mathcal{F}^\infty_{\alpha}}=\sup_{z\in\C}|f(z)|e^{-\frac{\a}{2}|z|^2}<\infty.$$
\begin{proposition}\label{pr:embedding}
	Let $\a>0$ and $\om\in \Ainfty$.
	Then:
	\begin{enumerate}
		\item \label{item:embedding1}
		If $0<p<q\le\infty$,
		\begin{equation*}
		\mathcal{F}^p_{\alpha}\subset\mathcal{F}^q_{\alpha}.
		\end{equation*}
		\item \label{item:embedding2} For any $\delta\in(0,\a)$,
		$$
		\mathcal{F}^\infty_{\alpha-\delta}\subset
		\mathcal{F}^p_{\a,\om}\subset \mathcal{F}^1_{\alpha+\delta}.
		$$
	\end{enumerate}
\end{proposition}
\begin{proof}
	 \eqref{item:embedding1} is well known (see for instance \cite[Theorem~2.10]{ZhuFock}).
	Let us prove (ii).
	Let
	$f\in \mathcal{F}^\infty_{\alpha-\delta}$.  By Lemma~\ref{le:2}
	$$\| f\|^p_{\mathcal{F}^p_{\a,\om}}\le  \|
	f\|^p_{\mathcal{F}^\infty_{\alpha-\delta}}\int_{\C}
	e^{-p\delta\frac{|z|^2}{2}}\om(z)\,dA(z)<\infty,
	$$
	which proves the first embedding.
	\par Now assume that $f\in\Fpaom$. Then, by Corollary~\ref{co:1} and
	Lemma~\ref{le:3}
	$$\| f\|_{\mathcal{F}^1_{\alpha+\delta}}\lesssim
	\|f\|_{\Fpaom}\int_{\C} e^{-\delta\frac{|z|^2}{2}}\om\left(D(z,1)
	\right)^{-1/p}\,dA(z)<\infty.
	$$
	This finishes the proof.
\end{proof}

The next result is a technical byproduct of Proposition \ref{pr:embedding}\eqref{item:embedding1} which
 will be useful to obtain appropriate estimates for  $0<p\le 1$, in terms of the ones for $p>1$.

\begin{corollary}\label{cor:subhP}
Let $g\in \mathcal{F}^{1}_{\beta}$, $0<\beta<2\alpha$. Then
for  $\theta\in (0,1]$ we have
	$
[P_\alpha^+(|g|)(z)]^\theta\lesssim P_{\theta \alpha}^+(|g|^\theta)(z).
$
\end{corollary}

\begin{proof}

	For $z\in\C$,
	\begin{align*}
	[P_\alpha^+(|g|)(z)]^\theta
	&\asymp \left[\int_\C \left|g(w)e^{\alpha w\overline{z}}\right|e^{-\alpha|w|^2}\,dA(w)
	\right]^\theta\\
	&\asymp\left\|g(\cdot)e^{\alpha (\cdot)\overline{z}}
	\right\|^\theta_{\mathcal{F}^1_{2\alpha}}
	\lesssim \left\|g(\cdot)e^{\alpha (\cdot)\overline{z}}
	\right\|^\theta_{\mathcal{F}^\theta_{2\alpha}}\\	
		& \asymp\int_\C |g(w)|^\theta |e^{\theta \alpha w\overline{z}}|e^{-\theta\alpha|w|^2}\,dA(w)
		\asymp P_{\theta \alpha}^+(|g|^\theta)(z).
	\end{align*}
	Note that the condition $\beta<2\alpha$ ensures that all the terms in the above inequalities are finite.
\end{proof}

\begin{lemma}\label{le:gooddefiF1}	
	Assume that $0<\b<2\a$  and $g\in\mathcal{F}^1_\b$. Then,  for each $z\in\C$:
	\begin{enumerate}
		\item \label{item:gooddefiF11} If $\gamma(\a,\b):=\frac{\a^2}{2\a-\b}$, then $P^+_\a(|g|)(z)\le e^{\frac{\gamma(\a,\b)}{2}|z|^2}\|g\|_{\mathcal{F}^1_\b}$;
		\item \label{item:gooddefiF12} For a non-negative integer $k$, $g^{(k)}(z)=\a^kP_\a(\overline{\zeta}^k g(\zeta))(z)$;
		\item \label{item:gooddefiF13} If $\gamma>\gamma(\a,\b)$, then
		$g^{(k)}\in \mathcal{F}^{1}_\gamma$.
	\end{enumerate}
\end{lemma}

\begin{proof}
	Since  $0<\b<2\a$, we have
	$$
	\mathcal{R}(\a\overline{z}\z)-(\a-\frac{\b}{2})|\z|^2\le \a|z||\z|-\frac{2\a-\b}{2}|\z|^2\le \frac{\a^2}{2(2\a-\b)}|z|^2.
	$$

	So,
	\begin{equation}\begin{split}\label{eq:e1}
	P^+_\a(|g|)(z)
 & = \int_{\C}|g(\z)|e^{-\frac{\b}{2}|\z|^2}
	e^{\mathcal{R}(\a\overline{z}\z)-(\a-\frac{\b}{2})|\z|^2}\,dA(\z)
	\\ &\le e^{\frac{\gamma(\a,\b)}{2}|z|^2}\|g\|_{\mathcal{F}^1_\b}.
	\end{split}\end{equation}
	Now, fix $z\in\C$ and take $\ep>0$  an a polynomial $\psi$ such that
	$$
	e^{\frac{\gamma(\a,\b)}{2}|z|^2} \|g-\psi\|_{\mathcal{F}^1_\b}<\frac{\ep}{2}\quad\text{and}\quad
	|g(z)-\psi(z)|<\frac{\ep}{2}.
	$$
	Then, by \eqref{eq:e1}
	\begin{equation*}\begin{split}
	\left| P_\a(g)(z)-g(z)\right|
	& \le \left| P_\a(g-\psi)(z)\right|+|\psi(z)
	-g(z)| \\ & \le P^+_\a(|g-\psi|)(z)+\frac{\ep}{2}
	\\ & \le e^{\frac{\gamma(\a,\b)}{2}|z|^2} \|g-\psi\|_{\mathcal{F}^1_\b}+\frac{\ep}{2}<\ep,
	\end{split}\end{equation*}
	which implies that $g(z)=P_\a(g)(z)$, for any $z\in\C$.
	From these results  it is easy to check  that
	$g^{(k)}(z)=\a^kP_\a(\overline{\zeta}^k g(\zeta))(z)$.
	
	Finally, observe that $|\zeta|^ke^{-\e |\zeta|^2}$ is a bounded function for any $\varepsilon>0$.
	Thus, $\zeta^k g(\zeta)\in \mathcal{F}^{1}_{\beta+\e}$
	for any $\varepsilon>0$. So, \eqref{item:gooddefiF13} is a consequence of \eqref{item:gooddefiF11} and \eqref{item:gooddefiF12}.
\end{proof}

The next lemma provides a representation formula of a holomorphic function in terms of its $k$-th derivatives.

\begin{lemma}\label{le:representation_formula}
	Let  $\beta\in(0,2\a)$, $k\in\N$ and $f\in \mathcal{F}^1_{\b}$. Let
	 $T_m(f)$ be the Taylor polynomial of degree $m$ for $f$, centered at $z = 0$.
	Then,
	\begin{equation*}
	f(z) = T_{2k-1}(f)(z) +R_{2k}(f)(z),
	\end{equation*}
	where
	$$
	R_{2k}(f)(z):= \frac{\a^{1-k}}{\pi}\int_{\C} \frac{(f(w)-T_{2k-1}(f)(w))^{(k)}}
	{\overline{w}^k}e^{\a z\overline{w}} e^{-\a|w|^2}\,dA(w).
	$$
	
	In particular, for $k=1$ we have
	\begin{equation*}
	f(z)=f(0)+f'(0)z+\frac1{\pi}\int_{\C}
	\frac{f'(w)-f'(0)}{\overline{w}}e^{\a z\overline{w}}
	e^{-\a|w|^2}dA(w).
	\end{equation*}
\end{lemma}

\begin{proof}
	By Lemma \ref{le:gooddefiF1} the function $g_m=f-T_m(f)$ and its $k$-th derivatives are in $\mathcal{F}^1_{\beta+\delta}$ for any $\delta>0$. Since $$(f(w)-T_{2k-1}(f)(w))^{(k)}\frac
	{w^k}{|w|^{2k}}$$
	is bounded on the unit disk, we obtain that this function is  in $\mathcal{L}^1_{\beta+\delta}$. So, $R_{2k}(f)$ is a entire function and
	\begin{equation*}\begin{split}
	&(f(w)-T_{2k-1}(f)(w))^{(k)}=\sum_{j=2k}^\infty \frac{f^{(j)}(0)}{j!} j(j-1)\dots (j-k+1)w^{j-k}\\&=w^k \sum_{m=0}^\infty \frac{f^{(m+2k)}(0)}{(m+2k)!} (m+2k)(m+2k-1)\dots (m+k+1)w^{m}.
	\end{split}\end{equation*}
	
	Next, orthogonality gives that
	\begin{equation*}\begin{split}
	&\frac{\a^{1-k}}{\pi}\int_{\C}
	\frac{w^{k+m}w^k}{|w|^{2k}}e^{\a z\overline{w}}e^{-\a|w|^2}
	dA(w)\\& =\frac{\a^{1-k}}{\pi}\int_{\C} |w|^{2k+2m} \frac{(\a
		z)^{2k+m}}{(2k+m)!}e^{-\a|w|^2}dA(w)\\&=
	\frac{2\a^{1+k+m}}{(2k+m)!}\left(\int_0^\infty r^{2k+2m+1}
	e^{-\a r^2}\,dr\right) z^{2k+m}\\&=
	\frac{\a^{k+m}}{(2k+m)!}\left(\int_0^\infty
	\frac{t^{k+m}}{\a^{k+m}}e^{-t}\,dt\right) z^{2k+m}=
	\frac{(k+m)!}{(2k+m)!}z^{2k+m},
	\end{split}\end{equation*}
	where in the previous to the last equality we have made the change
	of variables $t=\a r^2$. Bearing in mind  this calculation, the
	fact that $(f(w)-T_{2k-1}(f)(w))^{(k)} \frac{w^k}{|w|^{2k}}$ is bounded
	on $\D$, we deduce that
	\begin{equation*}\begin{split}
	&\frac{\a^{1-k}}{\pi}\int_{\C} \frac{(f(w)-T_{2k-1}(f)(w))^{(k)}}
	{\overline{w}^k}e^{\a z\overline{w}} e^{-\a|w|^2}dA(w)
	\\&= \sum_{m=0}^\infty (m+2k)\dots (m+k+1) \frac{f^{(m+2k)}(0)}{(m+2k)!}\frac{\a^{1-k}}{\pi}
	\int_{\C} \frac{w^{k+m}w^k}{|w|^{2k}}e^{\a z\overline{w}}e^{-\a|w|^2} dA(w)
	\\&=\sum_{m=0}^\infty\frac{f^{(m+2k)}(0)}{(m+2k)!}z^{m+2k}= f(z)-T_{2k-1}(f)(z),
	\end{split}\end{equation*}
	which gives the desired result.
\end{proof}

\begin{remark}\label{rem:valid}
	By Proposition \ref{pr:embedding}, if $\om\in\Ainfty$ then
	$\Fpaom\subset \mathcal{F}^1_{\a+\delta}$
	 for any $\delta>0$.
 So, all the theses in  Corollary \ref{cor:subhP}, Lemma \ref{le:gooddefiF1} and Lemma \ref{le:representation_formula}  remain true
replacing $\mathcal{F}^1_\beta$ by $\Fpaom$ in the corresponding hypotheses.
	\end{remark}

\subsection{Main results}
First, we will restrict ourselves to the case of the first derivative.

\begin{lemma}\label{lem:subhP}
	Let $p>0$ and $0<\beta<2\alpha$. If $\om\in \Ainfty$ and $g\in \mathcal{F}^{p}_{\b,\om}$,
then	$\|P_\a^+(|g|)\|_{\Lpaom}\lesssim \|g|\|_{\Fpaom}$.
\end{lemma}

\begin{proof}
	By Proposition \ref{prop:aprest}\eqref{item:aprest2}, we can assume that $\om\in A_{p_0}^{restricted}$ for some $p_0>\max\{p,1\}$.
 Remark~\ref{rem:valid}, Corollary  \ref{cor:subhP} with $\theta=p/p_0$ together  Theorem \ref{th:isra} give
	\begin{align*}
	\|P_\a^+(|g|)\|_{\Lpaom}&
	\lesssim \|
	(P_{p \alpha/p_0}^ +(|g|^{p/p_0}))^{p_0/p}\|_
	{\Lpaom}
	\asymp \|
	P_{p \alpha/p_0}^+(|g|^{p/p_0})\|_
	{\mathcal{L}^{p_0}_{\a p/p_0,\om}}\\
	&\lesssim \||g|^{p/p_0}\|_{\mathcal{L}^{p_0}_{\a p/p_0,\om}}\asymp
	\|g\|_{\Fpaom}.
	\end{align*}
	This ends the proof.
\end{proof}

\begin{proposition}\label{pr:in1}
	Let $0<p<\infty$, $\alpha>0$ and $\om\in \Ainfty$. Then,
	\begin{equation*}
	|f(0)|+\|f'\|_{\mathcal{F}^{p}_{\alpha,\omega_p}}
	\lesssim  \| f\|_{\Fpaom},\quad  f\in H(\C).
	\end{equation*}
\end{proposition}

\begin{proof}
	By Lemma~\ref{le:1}, there is $C=C(p,\om,\a)$ such that
	$$ |f(0)|^p\le C
	\int_{|z|\le 1}|f(z)|^pe^{-p\frac{\a}{2}|z|^2}\om(z)\,dA(z)\lesssim
	\| f\|^p_{\Fpaom}.$$

	So, to conclude the proof we need to show that
	$$
	\|f'\|_{\mathcal{F}^{p}_{\alpha,\omega_p}}
	\lesssim  \| f\|_{\Fpaom}.
	$$
	By Lemma~\ref{le:weightdistor},
	$\om_p(z)=\frac{\om(z)}{(1+|z|)^p}\in \Ainfty$. So,  by  Remark~\ref{rem:valid}, Lemma~\ref{le:gooddefiF1} and  Lemma~\ref{lem:subhP}
$$
\|f'\|_{\mathcal{F}^{p}_{\alpha,\omega_p}}
\lesssim \|P_\a^+(|wf(w)|)\|_{\mathcal{L}^{p}_{\alpha,\omega_p}} \lesssim \|wf(w)\|_{\mathcal{F}^{p}_{\alpha,\omega_p}} \lesssim\| f\|_{\Fpaom}.
$$
	This ends the proof.
\end{proof}

\begin{proposition}\label{pr:in2}
	Let $0<p<\infty$, $\alpha>0$ and $\om\in \Ainfty$. Then,
		\begin{equation*}
\| f\|_{\Fpaom}\lesssim |f(0)|+\|f'\|_{\mathcal{F}^{p}_{\alpha,\omega_p}},\quad  f\in H(\C).
	\end{equation*}
\end{proposition}

\begin{proof}
	By Lemma~\ref{le:weightdistor}, if $\om\in\Ainfty$, then $\om_p$ is also in $\Ainfty$.
		By Lemma~\ref{le:weightdistor} and Proposition~\ref{pr:embedding} ,	$f'\in \mathcal{F}^p_{\a,\om_{p}}\subset \mathcal{F}^1_{\a+\delta}$, for any
	$0<\delta<\a$. So, by \eqref{eq:LPfcla},  we have  $f\in \mathcal{F}^1_{\a+\delta}$. Thus, the representation formula in Lemma \ref{le:representation_formula} holds and Lemma~\ref{le:2} yield
	$$
\| f\|_{\Fpaom}\lesssim
|f(0)|+\|f'(0)z\|_{\Fpaom}+\|R_2(f)\|_{\Fpaom},
$$
where
$$
R_2(f)(z)=\frac{1}{\pi}\int_{\C}
\frac{f'(w)-f'(0)}{\overline{w}}e^{\a z\overline{w}}
e^{-\a|w|^2}dA(w).
	$$
	
By Lemma~\ref{le:1} and Proposition~\ref{pr:embedding},
	\begin{equation*}
|f'(0)|
\lesssim  \|f'\|_{\mathcal{F}^1_{\alpha+\delta}}\lesssim
	\|f'\|_{\mathcal{F}_{\a,\om_{p}}},\quad \delta>0.
	\end{equation*}	
	Since $\|z\|_{\Fpaom}<\infty$, we obtain
	$$
|f(0)|+\|f'(0)z\|_{\Fpaom}
\lesssim
|f(0)|+\|f'\|_{\mathcal{F}_{\a,\om_{p}}}.
	$$

	Now, let us deal with the last term above. 	
	Let $g(\z):=\frac{f'(\z)-f'(0)}{\z}$.
	Let us observe that if $|\z|\le 1$ and $\delta\in (0,\alpha)$, it follows from Proposition~\ref{pr:embedding} and Lemma~\ref{le:weightdistor}
\begin{equation*}
  \begin{split}
\left|g(\z) \right|
&= \left| \int_0^1 f^{''}(t\z)dt \right|\leq \sup_{|w|\leq 1}|f^{''}(w)|
\\
&\lesssim
	\int_{D(0,2)}|f'(w)|
e^{-{(\alpha+\delta)}|w|^2} \,dA(w)
\lesssim 	\|f'\|_{\mathcal{F}^1_{\a+\delta}}
\lesssim \|f'\|_{\mathcal{F}^p_{\a,\om_p}},
\end{split}\end{equation*}
	where in the
 second to last inequality we use that
	$e^{-{(\alpha+\delta)}|w|^2}\asymp 1$ on $D(0,2)$.

	Moreover, if $|\z|>1$,
	$$
\left|g(\z)\right|\le 2\frac{|f'(\z)|+|f'(0)|}{1+|\z|}
	\lesssim \frac{|f'(\z)|}{1+|\z|}+|f'(0)|.
	$$
	Thus, $g$ is a entire function which satisfies
	$$
	\left|g(\z) \right|\lesssim \frac{|f'(\z)|}{1+|\z|}
		+\|f'\|_{\mathcal{F}^p_{\a,\om_p}},
	$$
	which shows that  $g\in \Fpaom$ and
	\begin{equation}\label{eq:eqgf}
	\|g\|_{\Fpaom}\lesssim \|f'\|_{\mathcal{F}^p_{\a,\om_p}}.
	\end{equation}
	
Therefore, by Lemma \ref{lem:subhP}
	$$
	\|R_2(f)\|_{\Fpaom}
	\lesssim\|P_\a^+(|g|)(z)\|_{\Lpaom}
	\lesssim
	\|g\|_{\Fpaom}\lesssim \|f'\|_{\mathcal{F}^p_{\a,\om_p}}.
	$$	
	This ends the proof.
\end{proof}

\begin{Pf}{\em{Theorem~\ref{th:LPformula}.}}
	It follows from Propositions~\ref{pr:in1} and~\ref{pr:in2} that
	\begin{equation}\label{eq:LP1}
	\| f\|_{\Fpaom} \asymp
	|f(0)|+\|f'\|_{\mathcal{F}^{p}_{\alpha,\omega_p}}.
	\end{equation}

	By  Lemma~\ref{le:weightdistor}, $\om_{kp}=(\om_{(k-1)p})_p\in \Ainfty$ for any
	$k\in\N$. Therefore, consecutive iterations of \eqref{eq:LP1} give
	that
	$$\|f\|_{\Fpaom}\asymp
	\sum_{j=1}^{k-1}|f^{(j)}(0)|
	+\|f^{(k)}\|_{\mathcal{F}^{p}_{\alpha,\omega_{kp}}}.
	$$
	This finishes the proof.
\end{Pf}

\section{Carleson measures on weighted Fock-Sobolev spaces}\label{sec:Carleson}
We begin proving appropriate norm estimates  for the family of test functions employed  in the proof of Theorem~\ref{th:differentiation}.

\subsection{Test functions}
We will write $K_{\a,a}(z)=K_a(z)=e^{\a \overline{a}z}$ for the
reproducing kernels of the classical Fock space $\mathcal{F}^2_\a$.
\begin{proposition}\label{pr:kernelestimate}
	Let $\a,p\in (0,\infty)$ and $\om\in\Ainfty$. Then,
	\begin{equation*}
	\| K_a\|^p_{\Fpaom}\asymp e^{\a
		\frac{p|a|^2}{2}}\omega\left(D(a,1)\right),
\quad a\in\C.
	\end{equation*}
\end{proposition}
\begin{proof}
	Bearing in mind that
	
	\begin{equation*}\begin{split}
	\| K_a\|^p_{\Fpaom} &= e^{\a \frac{p|a|^2}{2}}\int_{\C}e^{\a
		\frac{-p|z-a|^2}{2}}\om(z)\,dA(z)= e^{\a
		\frac{p|a|^2}{2}}\int_{\C}e^{\a
		\frac{-p|u|^2}{2}}\om_{[a]}(u)\,dA(u),
	\end{split}\end{equation*}
	the proof follows from Lemma~\ref{le:translation} and
	Lemma~\ref{le:2}.
\end{proof}

\begin{proposition}\label{pr:kerneldeco}
	Let  $\a,p\in (0,\infty)$ and $\om\in\Ainfty$. Then, for any
	sequence $\{a_\nu\}_{\nu\in r\Z^2}\in l^p$,
	$$f(z)=\sum_{\nu\in r\Z^2} a_\nu \frac{K_{\nu}(z)}{\|
		K_\nu\|_{\Fpaom}}\in \Fpaom,$$ with $\|  f\|_{\Fpaom}\lesssim \|
	\{a_\nu\}\| _{l^p}$.
\end{proposition}
\begin{proof}
	If $0<p\le 1$, then
	\begin{equation*}\begin{split}
	\|  f\|^p_{\Fpaom}\le \int_{\C} \sum_{\nu\in r\Z^2}
	\frac{|a_\nu|^p|K_{\nu}(z)|^p}{\| K_\nu\|^p_{\Fpaom}} e^{\a
		\frac{-p|z|^2}{2}}\om(z)\,dA(z)= \sum_{\nu\in r\Z^2} |a_\nu|^p.
	\end{split}\end{equation*}
	If $1<p<\infty$, then bearing in mind
	Proposition~\ref{pr:kernelestimate},
	\begin{equation*}\begin{split}
	&\|  f\|^p_{\Fpaom} \\ &  \lesssim \int_{\C} \left(\sum_{\nu\in
		r\Z^2} \frac{|a_\nu|	|K_{\nu}(z)|}{e^{\a
			\frac{|\nu|^2}{2}}\om\left(D(\nu,1)\right)^{1/p}}
\right)^p e^{\a \frac{-p|z|^2}{2}}\om(z)\,dA(z)
	\\ & \lesssim \int_{\C} \left(\sum_{\nu\in r\Z^2}
	\frac{|a_\nu|^p	|K_{\nu}(z)|}{e^{\a \frac{|\nu|^2}{2}}\om\left(D(\nu,1)\right)}
\right) \left(\sum_{\nu\in r\Z^2}
	\frac{|K_{\nu}(z)|}{e^{\a \frac{|\nu|^2}{2}}} \right)^{p-1}
	e^{\a \frac{-p|z|^2}{2}}\om(z)\,dA(z).
	\end{split}\end{equation*}
	\par Next, by Lemma~\ref{le:1}
	\begin{equation*}\begin{split}
	\sum_{\nu\in r\Z^2} |K_{\nu}(z)|e^{-\a \frac{|\nu|^2}{2}} & \lesssim
	\sum_{\nu\in r\Z^2} \int_{D\left(\nu, \frac{r}{2} \right)}
	|K_{z}(u)|e^{-\a \frac{|u|^2}{2}}\,dA(u)
	\\ & \lesssim \int_{\C}|K_{z}(u)|e^{-\a \frac{|u|^2}{2}}\,dA(u)
	\asymp e^{\a \frac{|z|^2}{2}},\quad z\in\C.
	\end{split}\end{equation*}
	So, by Proposition~\ref{pr:kernelestimate}
	\begin{equation*}\begin{split}
	\|  f\|^p_{\Fpaom}  & \lesssim \sum_{\nu\in r\Z^2}
	\frac{|a_\nu|^p}{e^{\a \frac{|\nu|^2}{2}}\om\left(D(\nu,1)\right)}
	\int_{\C} |K_{\nu}(z)|
	e^{-\a \frac{|z|^2}{2}}\om(z)\,dA(z)
	\\ & \asymp \sum_{\nu\in r\Z^2}
	|a_\nu|^p.
	\end{split}\end{equation*}
	This finishes the proof.
\end{proof}

\subsection{Proof of Theorem \ref{th:differentiation}}
First, we will deal with the case $n=0$. Later on, we will use Theorem~\ref{th:LPformula} to deduce the rest of cases from this particular one.
\subsection{Case $p\leq q$}

\begin{theorem}\label{carlesonqmayorp}
	Let  $\a\in (0,\infty)$, $\om\in\Ainfty$ and let $\mu$ be a positive
	Borel measure on $\C$. For $0<p\le q<\infty$, the following
	conditions are equivalent;
	\begin{enumerate}
		\item \label{carlesonqmayorp1} $\mu$ is a $q$-Carleson measure for $\Fpaom$;
		\item \label{carlesonqmayorp2} $\| K_a\|_{L^q(\mu)}\lesssim \| K_a\|_{\Fpaom},
		\quad a\in\C$;
		\item \label{carlesonqmayorp3} $\int_{D(a,1)} e^{\a\frac{q}{2}|z|^2}\,d\mu(z)\lesssim \left(\om(D(a,1))\right)^{\frac{q}{p}},\quad
		a\in\C$.
	\end{enumerate}
	Moreover,
	\begin{equation*}
	\|I_d\|_{\Fpaom\to L^q(\mu)}^q\asymp \sup_{a\in \C}\frac{\|
		K_a\|^q_{L^q(\mu)}}{\| K_a\|^q_{L^p(\mu)}} \asymp \sup_{a\in
		\C}\frac{\int_{D(a,1)}
		e^{\a\frac{q}{2}|z|^2}\,d\mu(z)}{\left(\om(D(a,1))\right)^{\frac{q}{p}}}.
	\end{equation*}
\end{theorem}
\begin{proof}
	Let us denote $G:=\sup_{a\in
		\C}\frac{\int_{D(a,1)}
		e^{\a\frac{q}{2}|z|^2}\,d\mu(z)}{\left(\int_{D(a,1)}\om(z)\,dA(z)\right)^{\frac{q}{p}}}.$
	
	\par \eqref{carlesonqmayorp1}$\Rightarrow$\eqref{carlesonqmayorp2} and the inequality $\|I_d\|_{\Fpaom\to L^q(\mu)}^q\ge\sup_{a\in \C}\frac{\|
		K_a\|^q_{L^q(\mu)}}{\| K_a\|^q_{L^p(\mu)}}$ are clear.
	\par Assume \eqref{carlesonqmayorp2}. Then, by Proposition~\ref{pr:kernelestimate}, for any $a\in\C$,
	$$
	\int_{\C}e^{q\a\mathcal{R}(\overline{a}z)} e^{-\a\frac{q}{2}|a|^2}\,d\mu(z)=
	\int_{\C}e^{\a\frac{q}{2}|z|^2} e^{-\a\frac{q}{2}|z-a|^2}\,d\mu(z)
	\lesssim
	\left(\om(D(a,1))\right)^{\frac{q}{p}}.
	$$
So, for any $a\in\C$,
	$$
	\left(\om(D(a,1))\right)^{\frac{q}{p}}\gtrsim
	\int_{D(a,1)}e^{\a\frac{q}{2}|z|^2}
	e^{-\a\frac{q}{2}|z-a|^2}\,d\mu(z)\asymp \int_{D(a,1)}
	e^{\a\frac{q}{2}|z|^2}\,d\mu(z),
	$$
	 which gives \eqref{carlesonqmayorp3}
	and the inequality
	$$ \sup_{a\in \C}\frac{\|
		K_a\|^q_{L^q(\mu)}}{\| K_a\|^q_{L^p(\mu)}} \gtrsim \sup_{a\in
		\C}\frac{\int_{D(a,1)}
		e^{\a\frac{q}{2}|z|^2}\,d\mu(z)}{\left(\om(D(a,1))\right)^{\frac{q}{p}}}.$$
	\par Now, let
	us prove \eqref{carlesonqmayorp3}$\Rightarrow$\eqref{carlesonqmayorp1}.
	It follows from \eqref{carlesonqmayorp3} and Lemma \ref{le:1} that
	\begin{equation*}\begin{split}
	\int_{\C} |f(z)|^q&\,d\mu(z)\le
	\sum_{\nu\in\mathbb{Z}^2}\int_{D(\nu,1)} |f(z)|^q\,d\mu(z)
	\\ & = \sum_{\nu\in\mathbb{Z}^2}\int_{D(\nu,1)} \left|f(z)e^{-\a\frac{|z|^2}{2}}\right|^q\, e^{q\a\frac{|z|^2}{2}}\,d\mu(z)
	\\ & \lesssim \sum_{\nu\in\mathbb{Z}^2}\int_{D(\nu,1)}  \left(\frac{\int_{D(z,1)}|f(u)|^p e^{-p\a\frac{|u|^2}{2}}\om(u)\,dA(u)}
	{\om(D(z,1))}\right)^{\frac{q}{p}}
	e^{q\a\frac{|z|^2}{2}}\,d\mu(z).
	\end{split}\end{equation*}
	
	By Lemma \ref{le:isra34} (see also Remark \ref{rem:discoscomparables}),  Lemma \ref{le:2} and the fact that $D(z,1)\subset D(\nu,2)$ for $z\in D(\nu,1)$, we have
		\begin{equation*}\begin{split}
	&\int_{\C} |f(z)|^q\,d\mu(z)\\
 & \lesssim \sum_{\nu\in\mathbb{Z}^2} \left(\frac{\int_{D(\nu,2)}|f(u)|^p e^{-p\a\frac{|u|^2}{2}}\om(u)\,dA(u)}
	{\om(D(\nu,1))}\right)^{\frac{q}{p}} \int_{D(\nu,1)}
	e^{q\a\frac{|z|^2}{2}}\,d\mu(z)
	\\ & \lesssim G \sum_{\nu\in\mathbb{Z}^2}
	\left(\int_{D(\nu,2)}|f(u)|^p
	e^{-p\a\frac{|u|^2}{2}}\om(u)\,dA(u)\right)^{\frac{q}{p}}.
	\end{split}\end{equation*}
	Finally, using Minkowski's inequality and the fact that
	$\{D(\nu,2)\}_{\nu\in\mathbb{Z}^2}$ is a covering of $\C$ which
	overlaps finitely many times, it follows that
	$$ \|f\|^q_{L^q(\mu)}\lesssim  G  \|f\|^q_{\Fpaom}, \qquad f\in H(\C),$$
	which implies that $\|I_d\|_{\Fpaom\to L^q(\mu)}^q\lesssim G$ and
	finishes the proof.
\end{proof}

\subsection{Case $q< p$}

\begin{theorem}\label{carlesonqmenorp}
	Let be $\a\in (0,\infty)$, $\om\in\Ainfty$ and $\mu$ a positive
	Borel measure on $\C$. For $0<q<p<\infty$, the following conditions
	are equivalent;
	\begin{enumerate}
		\item \label{carlesonqmenorp1}$\mu$ is a $q$-Carleson measure for $\Fpaom$;
		\item \label{carlesonqmenorp2}The function
		$$H(u)= \frac{\int_{D(u,1)}
			e^{\frac{q\a |z|^2}{2}}\,d\mu(z)}{\om\left(D(u,1)\right)} \in
		L^{\frac{p}{p-q}}(\mathbb{C},\om).$$
	\end{enumerate}
	Moreover,
	\begin{equation*}
	\|I_d\|_{\Fpaom\to L^q(\mu)}^q\asymp
	\|H\|_{L^{p/(p-q)}(\mathbb{C},\om)}.
	\end{equation*}
\end{theorem}
\begin{proof} This proof uses ideas from \cite[Theorem~1]{L1}.
	Assume that \eqref{carlesonqmenorp2} holds. Then, by  Lemmas~\ref{le:1} and  \ref{le:isra34} (see also Remark \ref{rem:discoscomparables}), the equivalence
	\eqref{eq:discoscomparables} and H\"older's inequality, we obtain
	\begin{equation*}\begin{split}
	\int_{\C}|f(z)|^q\,d\mu(z) & \lesssim
	\int_{\C}\left[\frac{\int_{D(z,1)} |f(u)|^qe^{\frac{-q\a
				|u|^2}{2}}\om(u)\,dA(u)}{\om\left(D(z,2)\right)}\right] e^{\frac{q\a
			|z|^2}{2}}\,d\mu(z)
	\\ & \lesssim \int_{\C} |f(u)|^qe^{\frac{-q\a
			|u|^2}{2}}\left[ \frac{\int_{D(u,1)} e^{\frac{q\a
				|z|^2}{2}}\,d\mu(z)}{\om\left(D(u,1)\right)} \right] \om(u)\,dA(u)
	\\ & \le \|f\|^q_{\Fpaom} \int_{\C}\left[ \frac{\int_{D(u,1)} e^{\frac{q\a
				|z|^2}{2}}\,d\mu(z)}{\om\left(D(u,1)\right)} \right]^\frac{p}{p-q}
	\om(u)\,dA(u),
	\end{split}\end{equation*}
	which gives (i) and the inequality $$\|I_d\|_{\Fpaom\to
		L^q(\mu)}^q\lesssim \int_{\C}\left[ \frac{\int_{D(u,1)} e^{\frac{q\a
				|z|^2}{2}}\,d\mu(z)}{\om\left(D(u,1)\right)} \right]^\frac{p}{p-q}
	\om(u)\,dA(u).$$
	\par Reciprocally,
	assume that (i) holds and
let us consider the
	functions $$G_t(z)=\sum_{\nu\in \Z^2} a_\nu R_{\nu}(t)
	\frac{K_{\nu}(z)}{\| K_\nu\|_{\Fpaom}}$$ where  $R_ \nu(t)$ is a
	sequence of Rademacher functions (see page $336$ of \cite{L1}, or
	Appendix A of \cite{Duren}). So, using $G_t$ as test functions,
	applying Khinchine's inequality and using
	Proposition~\ref{pr:kerneldeco}, we deduce that
	\begin{equation}\begin{split}\label{eq:15}
	\int _{\C} \left( \sum_{\nu\in \Z^2}
	|a_\nu|^2\frac{|K_{\nu}(z)|^2}{\|K_\nu\|^2_{\Fpaom}}
	\right)^{\frac{q}{2}}\,d\mu(z)  & \asymp \int_{\C}\int_0^1 \left|
	G_t(z)\right|^q\,dt \,d\mu(z)
	\\ & \lesssim  \|I_d\|_{\Fpaom\to
		L^q(\mu)}^q \| a_\nu\|^q_{l^p}.
	\end{split}\end{equation}
	Next, bearing in mind that
	$\left\{D\left(\nu,2\right)\right\}_{\nu\in \Z^2}$ is a  covering of
	$\C$ which overlaps finitely many times,  and using
	Proposition~\ref{pr:kernelestimate} it follows that
	\begin{equation*}\begin{split}
	\int _{\C} &\left( \sum_{\nu\in r\Z^2}
	|a_\nu|^2\frac{|K_{\nu}(z)|^2}{\|K_\nu\|^2_{\Fpaom}}
	\right)^{\frac{q}{2}}\,d\mu(z)\\
	& \ge \int _{\C} \left( \sum_{\nu\in r\Z^2}
	|a_\nu|^2\frac{|K_{\nu}(z)|^2}{\|K_\nu\|^2_{\Fpaom}}\chi_{D\left(\nu,2\right)}(z)
	\right)^{\frac{q}{2}}\,d\mu(z)
	\\ & \gtrsim \int _{\C} \sum_{\nu\in \Z^2}
	|a_\nu|^q\frac{|K_{\nu}(z)|^q}{\|K_\nu\|^q_{\Fpaom}}\chi_{D\left(\nu,2\right)}(z)
	\,d\mu(z)
	\\ & =  \sum_{\nu\in \Z^2}
	\frac{|a_\nu|^q}{\|K_\nu\|^q_{\Fpaom}}\int _{D\left(\nu,2\right)}
	|K_{\nu}(z)|^q\,d\mu(z)
	\\ & \asymp \sum_{\nu\in \Z^2}
	\frac{|a_\nu|^q e^{-\frac{q\a
				|\nu|^2}{2}}}{\left(\om\left(D(\nu,1)\right)
		\right)^{\frac{q}{p}}}\int _{D\left(\nu,2\right)}
	|K_{\nu}(z)|^q\,d\mu(z)
	\\ &
	\asymp \sum_{\nu\in \Z^2}
	\frac{|a_\nu|^q }{\left(\om\left(D(\nu,1)\right)
		\right)^{\frac{q}{p}}}\int _{D\left(\nu,2\right)} e^{\frac{q\a
			|z|^2}{2}}\,d\mu(z).
	\end{split}\end{equation*}
	
	These estimates together with \eqref{eq:15} and the classical duality relation
	$(l^{\frac{p}{q}})^\star\simeq
	l^{\left(\frac{p}{q}\right)'}=l^{\frac{p}{p-q}}$, $p>q$, gives that
	\begin{equation*}\begin{split}
	\sum_{\nu\in \Z^2} \left(\frac{\int _{D\left(\nu,2\right)}
		e^{\frac{q\a |z|^2}{2}}\,d\mu(z) }{\om\left(D(\nu,1)\right)}
	\right)^{\frac{p}{p-q}} \om\left(D(\nu,1)\right)
	 & =\sum_{\nu\in \Z^2}
	\left(\frac{\int _{D\left(\nu,2\right)} e^{\frac{q\a
				|z|^2}{2}}\,d\mu(z) }{\left(\om\left(D(\nu,1)\right)
		\right)^{\frac{q}{p}}}\right)^{\frac{p}{p-q}}\\ &\lesssim
	\|I_d\|_{\Fpaom\to L^q(\mu)}^q.
	\end{split}\end{equation*}
	So, bearing in mind Remark~\ref{rem:discoscomparables}
	\begin{equation*}\begin{split}
	& \int_{\C}\left[ \frac{\int_{D(u,1)} e^{\frac{q\a
				|z|^2}{2}}\,d\mu(z)}{\om\left(D(u,1)\right)}
	\right]^\frac{p}{p-q}\,\om(u)\,dA(u)
	\\ & \lesssim \sum_{\nu\in \Z^2} \left[ \frac{\int_{D(\nu,2)} e^{\frac{q\a
				|z|^2}{2}}\,d\mu(z)}{\om\left(D(\nu,1)\right)} \right]^\frac{p}{p-q}
	\om\left(D(\nu,1)\right)\lesssim \|I_d\|_{\Fpaom\to L^q(\mu)}^q.
	\end{split}\end{equation*}
	This finishes the proof.
\end{proof}

\begin{Pf}{\em{Theorem~\ref{th:differentiation}.}}
	\par The case $n=0$ follows from Theorem~\ref{carlesonqmayorp} and
	Theorem~\ref{carlesonqmenorp}. Next, by Lemma~\ref{le:weightdistor}
	$$\om\in\Ainfty \Leftrightarrow
	\om_{\gamma}(z)=\frac{\om(z)}{(1+|z|)^{\gamma}}\in\Ainfty.$$
	So  the equivalece (ii)$\Leftrightarrow$(iii)  for   $p\le q$ follows from Theorem~\ref{carlesonqmayorp}
	and the same equivalence for $q<p$ follows from
	Theorem~\ref{carlesonqmenorp}.
	\par Now let us prove (i)$\Leftrightarrow$(ii). First, assume that $n\in\N$ and $I_d:
	\mathcal{F}^p_{\a,\om_{np}}\to L^q(\mu)$ is bounded, that is
	$$\| f\|_{L^q(\mu)}\le \|I_d\|_{\mathcal{F}^p_{\a,\om_{np}}\to L^q(\mu)}\|
	f\|_{\mathcal{F}^p_{\a,\om_{np}}},\quad f\in H(\C).$$ Then,
	applying the above inequality to $f^{(n)}$, for each $f\in H(\C)$, and
	using Theorem~\ref{th:LPformula} we deduce that
	\begin{equation*}\begin{split}
	\| f^{(n)}\|_{L^q(\mu)} & \le
	\|I_d\|_{\mathcal{F}^p_{\a,\om_{np}}\to L^q(\mu)}
	\|f^{(n)}\|_{\mathcal{F}^p_{\a,\om_{np}}}
	\\ & \le \|I_d\|_{\mathcal{F}^p_{\a,\om_{np}}\to L^q(\mu)} \left(
	\sum_{k=0}^{n-1}|f^{(k)}(0)|^p +
	\|f^{(n)}\|^p_{\mathcal{F}^p_{\a,\om_{np}}}\right)^{1/p}
	\\ & \asymp \|I_d\|_{\mathcal{F}^p_{\a,\om_{np}}\to
		L^q(\mu)}\|f\|_{\Fpaom},
	\end{split}\end{equation*}
	that is $D^{(n)}:\,\Fpaom\to L^q(\mu)$ is bounded and $\|
	D^{(n)}\|_{\Fpaom\to L^q(\mu)}\lesssim
	\|I_d\|_{\mathcal{F}^p_{\a,\om_{np}}\to L^q(\mu)}.$
	\par Reciprocally, assume that $D^{(n)}:\,\Fpaom\to L^q(\mu)$ is
	bounded, that is,
	$$\| f^{(n)}\|_{L^q(\mu)}\le \|
	D^{(n)}\|_{\Fpaom\to L^q(\mu)} \|f\|_{\Fpaom},\quad f\in H(\C).$$
	Then, by Theorem~\ref{th:LPformula},
	$$\| f^{(n)}\|_{L^q(\mu)}\le \|
	D^{(n)}\|_{\Fpaom\to L^q(\mu)} \left( \sum_{k=0}^{n-1}|f^{(k)}(0)|^p
	+ \|f^{(n)}\|^p_{\mathcal{F}^p_{\a,\om_{np}}}\right)^{1/p},\quad
	f\in H(\C).$$ Now, replacing in the above inequality $f$ by
	$D^{(-n)}(f)$, it follows that
	$$\| f\|_{L^q(\mu)}\le \|
	D^{(n)}\|_{\Fpaom\to L^q(\mu)}
	\|f\|_{\mathcal{F}^p_{\a,\om_{np}}},\quad f\in H(\C).$$ Therefore,
	$I_d: \mathcal{F}^p_{\a,\om_{np}}\to L^q(\mu)$ is bounded and $$
	\|I_d\|_{\mathcal{F}^p_{\a,\om_{np}}\to L^q(\mu)} \|\lesssim
	\|D^{(n)}\|_{\Fpaom\to L^q(\mu)}.$$
Consequently, we have already proved (i)$\Leftrightarrow$(ii)$\Leftrightarrow$(iii),
as well as \eqref{normidentityfs} and \eqref{normidentityfs2}, for any $n\in\N\cup\{0\}$.
	\medskip
	\par Second, assume that $n$ is a negative integer and
	$D^{(n)}:\,\Fpaom\to L^q(\mu)$ is bounded, that is,
	$$\| D^{(n)}(f)\|_{L^q(\mu)}\le \|
	D^{(n)}\|_{\Fpaom\to L^q(\mu)} \|f\|_{\Fpaom},\quad f\in H(\C).$$
	Then, applying the above inequality to $f^{(-n)}$  and using
	Theorem~\ref{th:LPformula} (with $\om_{np}$) it follows that
	\begin{equation*}\begin{split}
	\| f\|_{L^q(\mu)} &
	\lesssim \left(\| D^{(n)}\|_{\Fpaom\to L^q(\mu)}+C_{\mu,n} \right)
	\left( \sum_{k=0}^{-n-1}|f^{(k)}(0)|^p +
	\|f^{(-n)}\|^p_{\mathcal{F}^p_{\a,\om}}\right)^{1/p}
	\\ & \asymp  \left(\| D^{(n)}\|_{\Fpaom\to L^q(\mu)}+C_{\mu,n} \right)\|f\|_{\mathcal{F}^p_{\a,\om_{np}}}.
	\end{split}\end{equation*}
	So, $I_d: \mathcal{F}^p_{\a,\om_{np}}\to L^q(\mu)$ is bounded and $$
	\|I_d\|_{\mathcal{F}^p_{\a,\om_{np}}\to L^q(\mu)} \lesssim
	\|D^{(n)}\|_{\Fpaom\to L^q(\mu)}+C_{\mu,n}.$$
	\par Finally, assume that $I_d: \mathcal{F}^p_{\a,\om_{np}}\to L^q(\mu)$ is
	bounded, that is,
	$$\| f\|_{L^q(\mu)}\le \|I_d\|_{\mathcal{F}^p_{\a,\om_{np}}\to L^q(\mu)}\|
	f\|_{\mathcal{F}^p_{\a,\om_{np}}},\quad f\in H(\C).$$ Then,
	replacing in the above inequality $f$ by $D^{(n)}(f)$ and applying
	Theorem~\ref{th:LPformula} (with $\om_{np}$), it follows that
	\begin{equation*}\begin{split}
	\| D^{(n)}(f)\|_{L^q(\mu)}\le
	\|I_d\|_{\mathcal{F}^p_{\a,\om_{np}}\to L^q(\mu)}\|
	D^{(n)}(f)\|_{\mathcal{F}^p_{\a,\om_{np}}}\asymp
	\|I_d\|_{\mathcal{F}^p_{\a,\om_{np}}\to L^q(\mu)}
	\|f\|_{\mathcal{F}^p_{\a,\om}}.
	\end{split}\end{equation*}
	So, $D^{(n)}:\,\Fpaom\to L^q(\mu)$ is bounded and $\|
	D^{(n)}\|_{\Fpaom\to L^q(\mu)}\lesssim
	\|I_d\|_{\mathcal{F}^p_{\a,\om_{np}}\to L^q(\mu)}.$
Consequently, (i)$\Leftrightarrow$(ii)$\Leftrightarrow$(iii),
as well as \eqref{normidentityfsnegative} and \eqref{normidentityfs2negative} hold for any negative integer $n$. This finishes the proof.
\end{Pf}

\section{Pointwise multipliers and embeddings}
\label{sec:PM-E}
\par In this section, by using Theorem \ref{th:differentiation} (with $n=0$)   we provide descriptions of the space of pointwise multipliers
 $\mathcal{F}^{p}_{\alpha,\om}$ to
$\mathcal{F}^{q}_{\beta,\eta}$, where $0<q,p<\infty$ and $\om$ is a $\Ainfty$-weight.
Since the constant functions are in $\mathcal{F}^{p}_{\alpha,\omega}$, then
$$
Mult(\mathcal{F}^{p}_{\alpha,\om},
\mathcal{F}^{q}_{\beta,\eta})\subset \mathcal{F}^{q}_{\beta,\eta}.
$$
In fact, it follows from the closed graph theorem that

	$g\in Mult(\mathcal{F}^{p}_{\alpha,\om},
	\mathcal{F}^{q}_{\beta,\eta})$ if and only if
	$$d\mu_{g}(z):=|g(z)|^qe^{-q\frac{\beta}{2}|z|^2}\eta(z)dA(z)$$ is a $q$-Carleson measure for $\mathcal{F}^{p}_{\alpha,\om}$.
Thus,  Theorem~\ref{th:differentiation} yields the following result.

\begin{corollary}\label{cor:PM}
	Let $\alpha,\beta>0$, $\om\in\Ainfty$ and $\eta$ be a weight. If $g\in \mathcal{F}^{q}_{\beta,\eta}$, then
	\begin{enumerate}
		\item If $0<p\le q<\infty$,
		$g\in Mult(\mathcal{F}^{p}_{\alpha,\om}\to
		\mathcal{F}^{q}_{\beta,\eta})$ if and only if there exists $C>0$ such that
		$$
		G(u):=\frac{1}{\omega(D(u,1))}\int_{D(u,1)} |g(z)|^qe^{-q\frac{\beta-\alpha}2|z|^2}\eta(z)dA(z)\le C \om(D(u,1))^{(q-p)/p},
		$$
		for any $u\in\C$.
		\item If $0<q< p<\infty$,
		$g\in Mult(\mathcal{F}^{p}_{\alpha,\om},
		\mathcal{F}^{q}_{\beta,\eta})$
		if and only if
		$G\in L^{p/(p-q)}(\C,\om)$.
	\end{enumerate}
\end{corollary}

In particular, taking $g=1$ in the above corollary we obtain the next result.

\begin{corollary}\label{cor:E}
	Let $\alpha,\beta>0$, $\om\in\Ainfty$ and $\eta$ be a weight. Then,
	\begin{enumerate}
		\item If $0<p\le q<\infty$,
		$\mathcal{F}^{p}_{\alpha,\om}\subset
		\mathcal{F}^{q}_{\beta,\eta}$ if and only if there exists $C>0$ such that
		$$
		G_1(u):=\frac{1}{\omega(D(u,1))}\int_{D(u,1)} e^{-q\frac{\beta-\alpha}2|z|^2}\eta(z)dA(z)\le C\om(D(u,1))^{(q-p)/p},
		$$
		for any $u\in\C$.
		\item If $0<q< p<\infty$,
		$\mathcal{F}^{p}_{\alpha,\om}\subset
		\mathcal{F}^{q}_{\beta,\eta}$ if and only if
		$G_1\in L^{p/(p-q)}(\C,\om)$.
	\end{enumerate}
\end{corollary}
\par Despite the conditions in Corollary~\ref{cor:PM} are useful in praxis,
 let us see they are particularly neat when both weights coincide.

\begin{proposition}\label{pr:mult}
	If $\beta<\alpha$, then
	$ Mult(\mathcal{F}^{p}_{\alpha,\om},
	\mathcal{F}^{q}_{\beta,\om})={0}$.
\end{proposition}

\begin{proof}
	Let $\delta\in (0,\min\left\{\alpha, \frac{\alpha-\beta}{2}\right\})$.
	By Proposition~\ref{pr:embedding},
	$\mathcal{F}^\infty_{\a-\delta}\subset
	\mathcal{F}^p_{\a,\om}$ and
	$	\mathcal{F}^{q}_{\beta,\om}\subset \mathcal{F}^\infty_{\b+\delta}$,
	so
	$$
	Mult(\mathcal{F}^{p}_{\alpha,\om},
	\mathcal{F}^{q}_{\beta,\om})\subset
	Mult(\mathcal{F}^{\infty}_{\alpha-\delta},
	\mathcal{F}^{\infty}_{\beta+\delta}).
	$$
	
	For each $a\in\C$ consider the function $f_a(z)=e^{(\a-\delta)z\overline{a}}$.
Since
	$$
	|f_a(z)|e^{-(\a-\delta)|z|^2/2}\le e^{(\a-\delta)|a|^2/2}=\|f_a\|_{\mathcal{F}^{\infty}_{\alpha-\delta}},
	$$
	for any $g\in Mult(\mathcal{F}^{\infty}_{\alpha-\delta},
	\mathcal{F}^{\infty}_{\beta+\delta})$,
	\begin{align*}
	|g(a)|e^{(2\a-\b-3\delta)|a|^2/2}
	&=|g(a)f_a(a)|e^{-(\b+\delta)|a|^2/2}
	\le \|gf_a\|_{\mathcal{F}^{\infty}_{\beta+\delta}}\\
	&\lesssim
	\|f_a\|_{\mathcal{F}^{\infty}_{\a-\delta}} =e^{(\a-\delta)|a|^2/2},
	\end{align*}
	which gives $|g(a)|\lesssim e^{(\b-\alpha+2\delta)|a|^2/2}$.
	
	Since $\b-\alpha+2\delta<0$, the maximum modulus principle gives $g=0$.
\end{proof}

\begin{theorem}\label{thm:PMom}
	
	Let $\alpha,\beta>0$ and $\om\in\Ainfty$. Then,
	\begin{enumerate}
		\item \label{item:PMom1} If $\alpha<\beta$, then  $Mult(\mathcal{F}^{p}_{\alpha,\om},
		\mathcal{F}^{p}_{\beta,\om})=\mathcal{F}^\infty_{\beta-\alpha}$;
		\item \label{item:PMom2} 
 $Mult(\mathcal{F}^{p}_{\alpha,\om},
		\mathcal{F}^{p}_{\alpha,\om})=\C$;
\item \label{item:PMom5} If $q>p$ and $\alpha\le\beta$, then $g\in Mult(\mathcal{F}^{p}_{\alpha,\om},\mathcal{F}^{q}_{\beta,\om})$ if and only if
		$$
		|g(u)|e^{\frac{\a-\b}2|u|^2}\lesssim \om(D(u,1))^{(q-p)/{pq}};$$
		\item \label{item:PMom3} If $p>q>0$ and $\alpha<\beta$, then  $Mult(\mathcal{F}^{p}_{\alpha,\om},\mathcal{F}^{q}_{\beta,\om})=\mathcal{F}^{pq/(p-q)}_{\beta-\alpha,\om}$;
		\item \label{item:PMom4} If $p>q>0$,
 then  $Mult(\mathcal{F}^{p}_{\alpha,\om}, \mathcal{F}^{q}_{\alpha,\om})=H(\C)\cap L^{pq/(p-q)}(\C, \om)$.
		
	\end{enumerate}
\end{theorem}

\begin{proof}
	Bearing in mind Corollary \ref{cor:PM}, in each case (i)-(v), it is easy to see that the condition describing $Mult(\mathcal{F}^{p}_{\alpha,\om}, \mathcal{F}^{q}_{\alpha,\om})$  
is   sufficient.
	Indeed, in order to prove this implication in  \eqref{item:PMom1}, \eqref{item:PMom2} and \eqref{item:PMom5},
	it is enough to use these conditions
 to prove that the corresponding function $G$
	satisfies the inequality in Corollary \ref{cor:PM}(i).  In  \eqref{item:PMom5},  we also use that
	$\omega(D(z,1))\asymp \omega(D(u,1))$
	for any $z\in D(u,1)$ (see Lemma \ref{le:isra34}). The analogue implications in
\eqref{item:PMom3}, \eqref{item:PMom4} follow from Corollary~\ref{cor:PM}(ii) and H\"older's inequality with exponent $p/q$.
	
	In each case (i)-(v), the
reverse direction
follows  from Corollary \ref{cor:PM} and Lemma \ref{le:1}, which provides the pointwise estimate
	$$
	|g(u)|e^{\frac{\a-\b}2|u|^2}\lesssim G(u)^{1/q},\qquad \beta\ge\alpha.
	$$
In particular, this estimate in  \eqref{item:PMom2} gives $Mult(\mathcal{F}^{p}_{\alpha,\om},
	\mathcal{F}^{p}_{\alpha,\om})\subset H(\C)\cap L^\infty$,
	which, by Liouville's theorem, coincides with $\C$.
\end{proof}

It is worth mentioning that Proposition~\ref{pr:mult} and Theorem~\ref{thm:PMom}, together with Theorem~\ref{th:LPformula}, provide descriptions
of pointwise multipliers  between Fock-Sobolev spaces induced by $\Ainfty$-weights.

\end{document}